\documentclass[12pt]{amsart}

\usepackage{amssymb}
\usepackage{mathabx}
\usepackage{bbm}
\usepackage{amsthm}
\usepackage{dsfont}
\usepackage{amsmath}
\usepackage{color,graphics,srcltx}
\usepackage{bm}
\usepackage{easybmat}
\usepackage{multirow,bigdelim}
\usepackage{enumitem}
\usepackage{graphicx}

\renewcommand{\dot}[1]{{{#1}^\sigma}}

\newtheorem{proposition}{Proposition}
\newtheorem{example}[proposition]{Example}
\newtheorem{lemma}[proposition]{Lemma}
\newtheorem{corollary}[proposition]{Corollary}
\newtheorem{theorem}[proposition]{Theorem}

\newcommand{\bPhi}{{\bf \Phi}}
\newcommand{\bu}{{\bf u}}

\def\qed{\hfill$\square$}

\begin{document}

\title{Asymmetric linkages: maxmin vs.\ reflected maxmin copulas}

\author[D. Kokol Bukov\v sek]{Damjana Kokol Bukov\v{s}ek}
\address{Damjana Kokol Bukov\v{s}ek, School of Economics and Business, University of Ljubljana, and Institute of Mathematics, Physics and Mechanics, Ljubljana, Slovenia}
\email{damjana.kokol.bukovsek@ef.uni-lj.si}

\author[T. Ko\v sir]{Toma\v{z} Ko\v{s}ir}
\address{Toma\v{z} Ko\v{s}ir, Faculty of Mathematics and Physics, University of Ljubljana, and Institute of Mathematics, Physics and Mechanics, Ljubljana, Slovenia}
\email{tomaz.kosir@fmf.uni-lj.si}

\author[B. Moj\v skerc]{Bla\v{z} Moj\v{s}kerc}
\address{Bla\v{z} Moj\v{s}kerc, School of Economics and Business, University of Ljubljana, and Institute of Mathematics, Physics and Mechanics, Ljubljana, Slovenia}
\email{blaz.mojskerc@ef.uni-lj.si}

\author[M. Omladi\v c]{Matja\v{z} Omladi\v{c}}
\address{Matja\v{z} Omladi\v{c}, Institute of Mathematics, Physics and Mechanics, Ljubljana, Slovenia}
\email{matjaz@omladic.net}
\begin{abstract}
    In this paper we introduce some new copulas emerging from shock models. It was shown in \cite{KoOm} that reflected maxmin copulas (RMM for short) are not just some specific singular copulas; they contain many important absolutely continuous copulas including the negative quadrant dependent ``half'' of the Eyraud-Farlie-Gumbel-Morgenstern class. The main goal of this paper is to develop the RMM copulas with dependent endogenous shocks and give evidence that RMM copulas may exhibit some characteristics better than the original maxmin copulas (MM for short): (1) An important evidence for that is the iteration procedure of the RMM transformation which we prove to be always convergent and we give many properties of it that are useful in applications. (2) Using this result we find also the limit of the iteration procedure of the MM transformation thus answering a question proposed in \cite{DuOmOrRu}. (3) We give the multivariate dependent RMM copula that compares to the MM version given in \cite{DuOmOrRu}. In all our copulas the idiosyncratic and systemic shocks are combined via asymmetric linking functions as opposed to Marshall copulas where symmetric linking functions are used.
\end{abstract}

\thanks{DKB\&TK\&BM\&MO acknowledge financial support from the Slovenian Research Agency (research core funding No. P1-0222).}
\keywords{Copula; dependence concepts; maxmin copulas; transformations of distribution functions; PQD property; survival analysis; shock models}
\subjclass[2010]{Primary: 	60E05; Secondary: 60E15, 62N05}

\maketitle

\section{ Introduction }

Dependence concepts play a crucial role in multivariate statistical literature since it was recognized that the independence assumption cannot describe conveniently the behavior of a stochastic system. One of the main tools have eventually become copulas due to their theoretical omnipotence emerging from the Sklar's theorem \cite{Skla}. 
There has been a vast literature on the subject 
including \cite{GeNe,KlMeSpSt,JwBaDeMe14,JwBaDeMe15,KlLiMePa}; for an excellent overview of these methods and the properties of copulas see \cite{Nels} and a more recent monograph \cite{DuSe}. For an overview of probability background, see for instance \cite{FrGr}.

In this paper we devote our study to the copulas emerging from shock models which are playing an important role in applications (cf.\  \cite{Hu,Mu,LiMcNe} and many more). This theory starts with Marshall and Olkin \cite{MaOl}, although it is Marshall \cite{Mars} who really introduces copulas into the picture of shock models. We believe that the third milestone on this path was set by Durante, Girard, and Mazo \cite{DuGiMa}, cf.\ also more generally \cite{ChDuMu}. The paper  \cite{DuGiMa} may have encouraged a vivid interest in the area \cite{ChMu,Mu,OmRu,DuOmOrRu,Hu, KoOm}. We want to point out the paper \cite{OmRu} where an asymmetric version of Marshall copulas was introduced, called \emph{maxmin copulas}. The dependent version of these copulas together with the multivariate version were introduced in \cite{DuOmOrRu}, while \cite{KoOm} brings into the picture the notion of \emph{reflected maxmin copulas, RMM} for short; it is also shown there that these copulas are always bounded above by the product copula $\Pi$ which means in particular that they are negatively quadrant dependent.

These views bring RMM copulas into the spotlight of some previously studied approaches and widen the range of their applications. The copulas of this type appeared in \cite[Proposition 3.2]{DuJa}. Moreover, these copulas may be viewed as perturbations of the product copula $\Pi$. General perturbations of copulas were studied in \cite{DuFeSaUbFl} and \cite{MeKoKo}, where a subclass of what we call reflected maxmin copulas were considered (cf.\ \cite[\S 3]{MeKoKo}). The RMM copulas can also be related to the results presented in \cite[Theorem 7.1]{DuMePaSe} with the maximum replaced by the minimum; we believe that this similarity is more than just coincidental. The copulas introduced in \cite{RoLaUbFl} are not only very close to our RMM copulas, as it was said in \cite{KoOm}, it turns out in our Example \ref{EFGM}\textbf{(a)} that they basically belong to MM copulas. Let us also point out that the maxmin (and consequently reflected maxmin) copulas have some properties that are appealing in various contexts related to the fuzzy set theory and the multicriteria decision making. The class includes nonsymmetric copulas that are used, for instance, as more general fuzzy connectives \cite{BeBoCiSaPlSa,DuKoMeSe}.

As an example consider two endogenous shocks $X_1$ and $X_2$ of a system, and one exogenous shock $Z$. Let the dependence of $(X_1, X_2)$ be governed by a copula $C$, while $Z$ is independent of it. The distribution of
\[
  (Y_1,Y_2) = (\max(X_1,Z), \max(X_2,Z)),
\]
as we know, is given by one of the two forms of the Marshall copula \cite{Mars}, while the other one is obtained from this one by replacing both linking functions $\max$ with the linking functions $\min$. The point is that the systemic shock has either an effect on both components coherent with the two idiosyncratic shocks, or it is conflictive with the two shocks.

On the other hand, if we take asymmetric linking functions (terminology introduced in \cite{DuGiMa}) $\max$ and $\min$, we get the maxmin copula as introduced in \cite{OmRu}. This may be viewed as if the shock $Z$ has opposite effects on the two components $X_1$ and $X_2$. So, it has a beneficial effect on one component (as seen through the linking function max) and detrimental effect on the other one (as seen through the linking function min), so that
\[
  (Y_1,Y_2) = (\max(X_1,Z), \min(X_2,Z)).
\]

We may think of $X_1$ and $X_2$ as r.v.'s representing the respective wealth of two groups of people, and the exogenous shock $Z$ is interpreted as an event that is beneficial to one of the groups and detrimental to the other one. Analogously, $X_1$ and $X_2$ can be thought of as a short and a long investment, respectively, while $Z$ is beneficial only to one of these types of investment. However, following \cite{KoOm} we argue that in this case there is a conceptual reason to replace one of the distribution functions corresponding to idiosyncratic shocks with the corresponding survival function. This way we get again the situation in which the systemic shock works in the same (or the opposite) direction as the idiosyncratic shock, but on both components in the same way.

The first among the important new results of this paper is the construction of the dependent version of the reflected maxmin copulas (RMM for short); the bivariate case is derived in Section \ref{sec:QSI}, cf.\ especially Formula \eqref{inverse_maxmin_final}, the multivariate extension of these copulas in Section \ref{sec:multi}, cf.\ Theorem \ref{thm:multi}; the multivariate form of the original maxmin copulas was given in \cite{DuOmOrRu}. At the end of Section \ref{sec:QSI} we slightly extend the result of \cite{KoOm} by showing that every copula of Eyraud-Farlie-Gumbel-Morgenstern class belongs to either the MM or the RMM class (cf.\ Example \ref{EFGM}\textbf{(b)}) and these are clearly only a very specific subclass of the MM and RMM copulas.  There is some evidence for the claim that RMM copulas are conceptually a better and more natural approach to study MM copulas. An important result in this direction is the iteration procedure of the RMM transformation which we prove to be always convergent and give many properties of it that are useful in applications. Our main result in this direction is Theorem \ref{thm:main} which presents the limit copula depending on the starting copula $C$ (governing the dependence of the idiosyncratic shocks), and the two functions $f$ and $g$ (``generators'' of the one step reflected maxmin transformation) together with the many dependence properties of the so constructed models presented in Section \ref{sec:properties}. Even more importantly, perhaps, we give as a consequence of these results the limit of the maxmin transformation in Theorem \ref{thm:maxmin} thus answering the question proposed in \cite{DuOmOrRu}.

In Section \ref{sec:QSI} we give an overview of the results on reflected maxmin copulas, extend them to the dependent case, give some properties and define the iterative RMM transformation. In Section \ref{sec:iteration} we actually perform the iteration procedure and prove that it is always convergent. This enables us to study in Section \ref{sec:properties} how dependence properties are (dis)inherited when this transformation is applied to a copula. The multivariate case of an RMM copula is first given for 3-variate case for the benefit of the reader in Section \ref{sec:multi3} and finally for the general case in Section \ref{sec:multi}.

\section{  Reflected maxmin copulas for dependent shocks }\label{sec:QSI}

In this section we introduce a dependent version of the bivariate reflected maxmin copulas presented in \cite{KoOm}. We start with two idiosyncratic shocks $X_1$ and $X_2$ and one systemic shock $Z$. We are seeking for the distribution of
\[
  (Y_1,Y_2) = (\max(X_1,Z), \min(X_2,Z)).
\]
Furthermore, we assume that the exogenous shock $Z$, having d.f.\ $G$, is independent of endogenous shocks $(X_1, X_2)$, while the joint d.f.\ of this vector can be expressed as $C(F_1, F_2)$, where $C$ is a given copula. The authors of \cite{DuOmOrRu}, where a dependent version of maxmin copulas is presented, introduce functions $\phi,\psi,\phi^*,\psi_*$ that help express the dependence of $(Y_1, Y_2)$ via a maxmin copula (let us point out that notation we are using here is slightly different from the notation in \cite{DuOmOrRu})
\begin{equation}\label{maxmin}
    T_{\phi,\psi}(C)(u,v) := u+(C(\phi(u),\psi(v)) - \phi(u)) \max\{0,\phi^*(u)- \psi_*(v)\}.
\end{equation}
Here we recall the sets of functions given in \cite{DuOmOrRu} that contain the generating functions $\phi,\psi$ and the corresponding auxiliary functions $\phi^*,\psi_*$: (1) $\mathcal{F}_1$ is the class of nondecreasing functions $\phi\colon [0,1] \rightarrow[0,1]$ such that $\phi(0) =0,\phi(1) =1$ and the function $\phi^* := \mathrm{id}/\phi$ is nondecreasing on $(0, 1]$; and (2) $\mathcal{F}_2$ is the class of nondecreasing functions $\psi\colon[0,1] \rightarrow[0,1]$ such that $\psi(0) =0,\psi(1) =1$ and the function
\[
\psi_*(v) := \left\{
               \begin{array}{ll}
                 \frac{v-\psi(v)}{1-\psi(v)}, & \hbox{if $v \in [0, 1)$;} \\
                 1, & \hbox{if $v = 1$.}
               \end{array}
             \right.
\]
is nondecreasing.\\

We want to develop the reflected maxmin copula from \eqref{maxmin} using the ideas and notation of \cite{KoOm}. First we replace $C(u,v)$ by the copula obtained from $C$ by one flip in the second variable. Following \cite{KoOm} we denote the so obtained copula by $\dot{C}$:
\[
    C(u,v)\mapsto \dot{C}(u,v):=u-C(u,1-v)
\]
to get
\[
    T_{\phi,\psi}(\dot{C})(u,v) := u- \dot{C}(\phi(u),1-\psi(v)) \max\{0,\phi^*(u)- \psi_*(v)\}.
\]
By performing the flip on the result of this transformation we obtain
\begin{equation}\label{inverse_maxmin}
    \dot{T_{\phi,\psi}}(\dot{C})(u,v) := \dot{C}(\phi(u),1-\psi(1-v)) \max\{0,\phi^*(u)- \psi_*(1-v)\}.
\end{equation}

The authors of \cite{DuOmOrRu} were studying transformation $C\mapsto T_{\cdot,\cdot}(C)$ (in our notation) and how the properties of copula $C$ are inherited when this transformation is applied to it. Our aim is to do that for transformation $\dot{C}\mapsto \dot{T_{\cdot,\cdot}}(\dot{C})$. Following \cite{KoOm} we replace the ``generating functions'' $\phi,\psi$ of the maxmin copula with the ``generating functions'' $f,g$ of the reflected maxmin copula and introduce the auxiliary functions $f^*,g^*,\widehat{f},\widehat{g}$:
\begin{align}\label{inverse_generators}
\begin{split}
   f(u) = \phi(u) - u,\ & g(v) = 1 - v - \psi(1 - v), \\
   f^*(u) = \frac{f(u)}{u}, \ \widehat{f}(u)=u+f(u),\   & g^*(v) = \frac{g(v)}{v}, \ \widehat{g}(v)=v+g(v).
\end{split}
\end{align}
A short computation translates formula \eqref{inverse_maxmin} into
\begin{equation}\label{inverse_maxmin_final}
    \dot{T_{f,g}}(\dot{C})(u,v) := uv \frac{\dot{C}(\widehat{f}(u), \widehat{g}(v))} {\widehat{f}(u)\widehat{g}(v)} \max\{0,1- f^*(u) g^*(v)\}.
\end{equation}
This is our extension to the dependent case of the formula \cite[(3)]{KoOm}. As a test, if we put into this formula $C=\Pi$, we first get $\dot{C}=\Pi$, so that \eqref{inverse_maxmin_final} is transformed to \cite[(3)]{KoOm} back.\\

It turns out that the generating and auxiliary functions given in \eqref{inverse_generators} satisfy conditions introduced in \cite{KoOm}
\begin{description}
  \item[(G1)] $f(0)=g(0)=0$, $f(1)=g(1)=0$, $f^*(1)=g^*(1)=0$,
  \item[(G2)] the functions $\widehat{f}(u)=f(u)+u$ and $\widehat{g}(u)=g(u)+u$ are nondecreasing on $[0,1]$.
  \item[(G3)] the functions $f^*$ and $g^*$ are nonincreasing on $(0,1]$.
\end{description}
Actually, it was shown in \cite[Lemma 1]{KoOm} and \cite[Lemma 2]{KoOm} that these conditions are equivalent to the starting conditions \textbf{(F1)}--\textbf{(F3)} on functions  $\phi,\psi$ studied in \cite{OmRu,DuOmOrRu}.

\begin{proposition}
  If $C$ is an arbitrary copula, and $f$ and $g$ satisfy conditions {\rm\textbf{\textbf{(G1)}}}, {\rm\textbf{(G2)}}, and {\rm\textbf{(G3)}}, then $\dot{T_{f,g}}(\dot{C})$ defined by \eqref{inverse_maxmin_final} is a copula.
\end{proposition}

\begin{proof}
  This follows immediately from \cite[Theorem 2.7]{DuOmOrRu} by considerations above after taking into account the fact that flipping of one of the variables sends copula to a copula.
\end{proof}


We present in Figure 1 
the scatterplots of the copulas obtained when transformation $\dot{T}$ is applied to three typical examples of copulas, $\Pi,M$, and $W$ in place of $C$. The choice of generating functions in this example is $f(u)= u(1-u)$ and $g(v)=\displaystyle \frac{1}{2}v(1-v)$. Further remarks on Figure 1 will be given after Proposition \ref{nqd}.

\textbf{Comment.} All scatterplots in the paper are drawn using Mathematica software \cite{math}. In the bivariate case we use the exact expressions for copulas and an algorithm to generate samples from a given copula presented in \cite[Chapter 2.9]{Nels}, i.e.\ (i) generate two independent uniform $(0, 1)$ variates $u$ and $t$; (ii) set $v = C_u^{(-1)}(t)$, where $C_u^{(-1)}$ denotes a quasi-inverse of the derivative $C_u$; (iii) the sample item is $(u,v)$. \qed

\begin{figure}[h!]\label{fig:slika1}
            \includegraphics[width=0.32\textwidth]{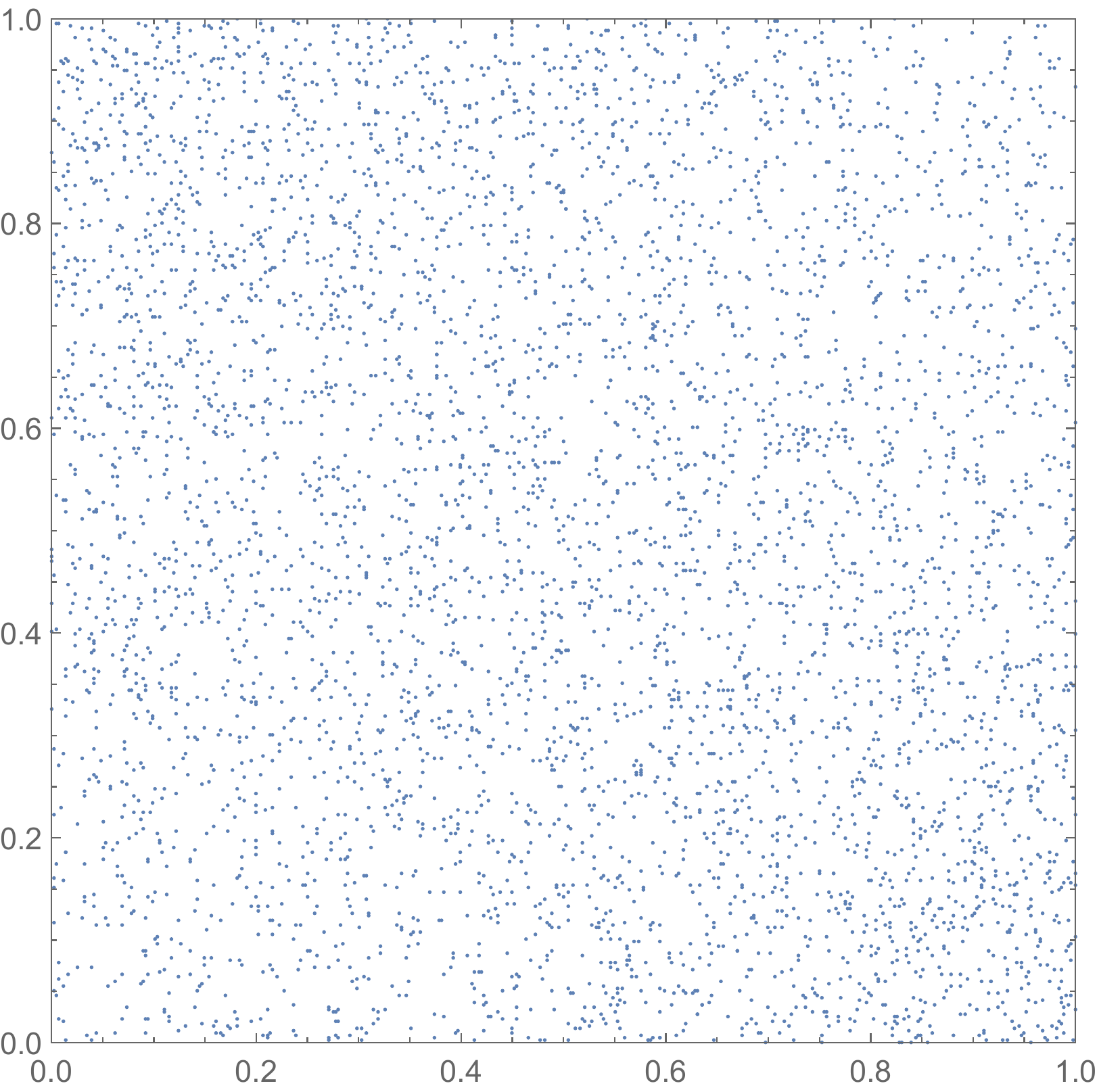} \hfil \includegraphics[width=0.32\textwidth]{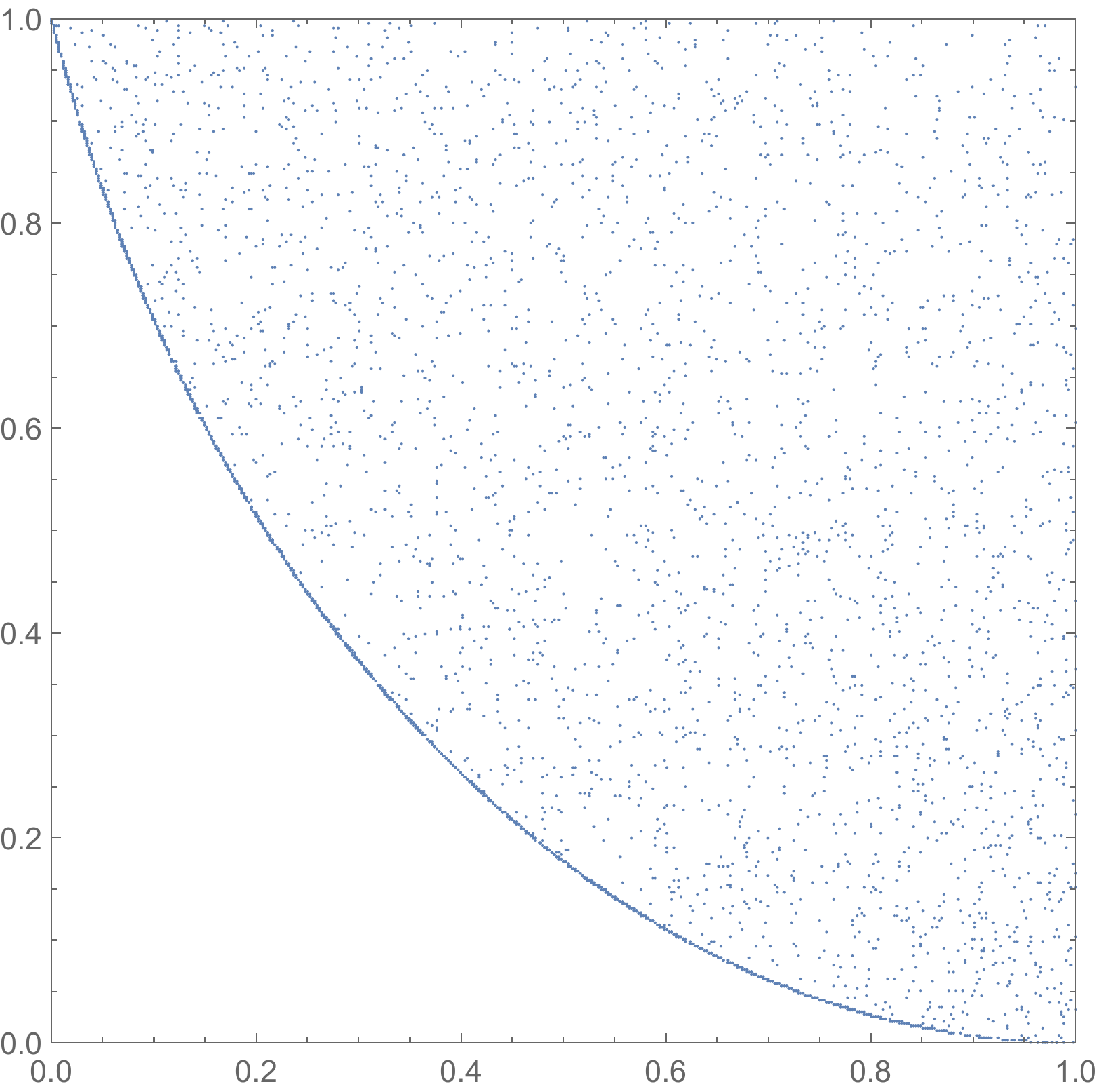} \hfil \includegraphics[width=0.32\textwidth]{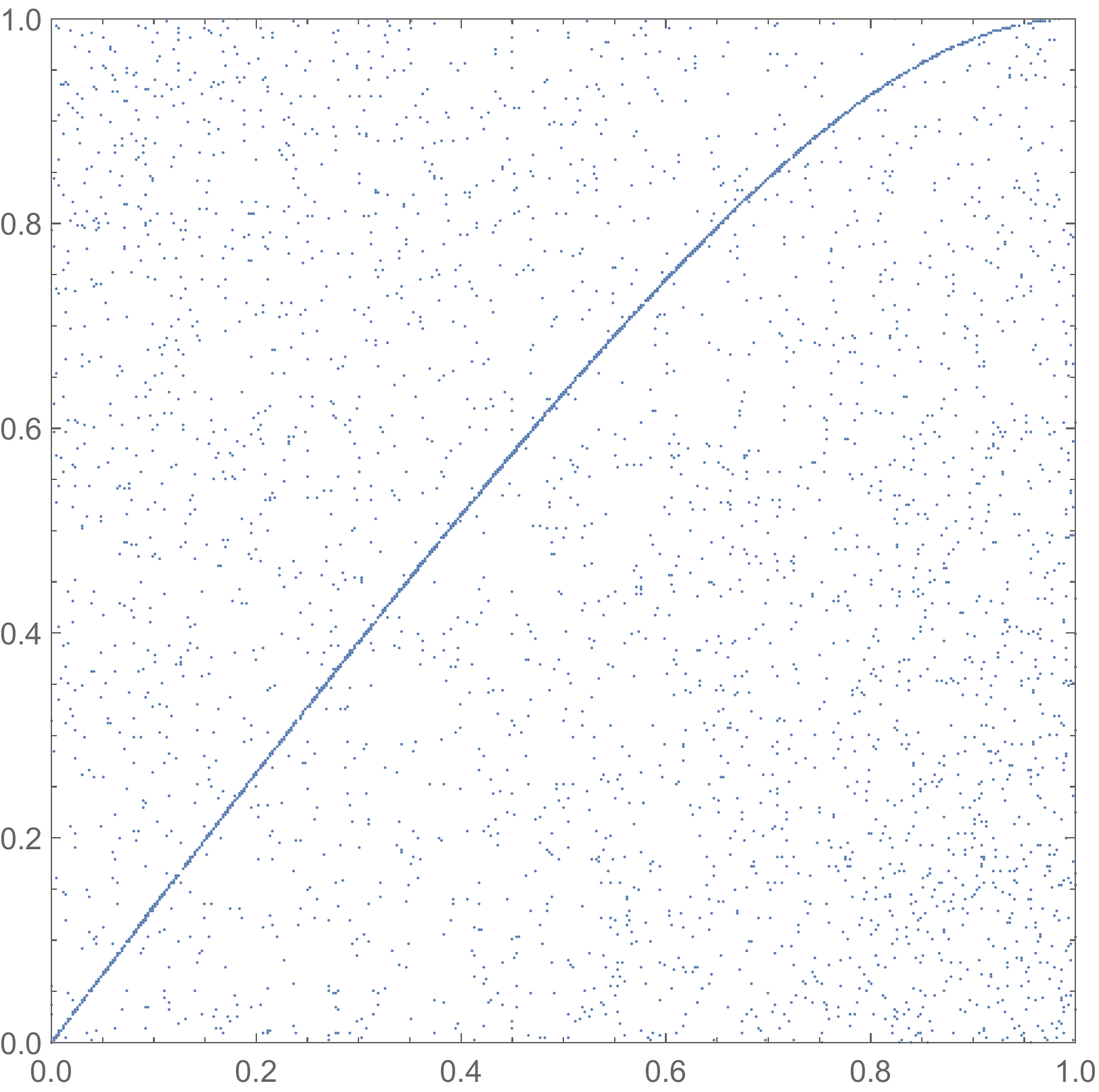}
            \caption{ Transformation $\dot{T}$ on copulas $\dot{\Pi},\dot{M}$, and $\dot{W}$, respectively }
\end{figure}

It has been observed that the formulas for reflected maxmin copulas \cite[(2)\&(3)]{KoOm} are conceptually clearer to the ones for maxmin copulas \cite[(1)]{KoOm} (rewritten from \cite[(2)]{OmRu}) in the case that the shocks are independent. We have now seen that the same is obviously true for the case of dependent shocks when we compare formula \eqref{inverse_maxmin_final} to formula \eqref{maxmin} (rewritten from \cite[(2.3)]{DuOmOrRu}).

At this point we want to present the first one of the properties that is being inherited by the transformation $\dot{C}\mapsto \dot{T_{\cdot,\cdot}}(\dot{C})$. Here we use the abbreviation PQD (respectively NQD) for the property that a random vector $(Y_1, Y_2)$ (or copula $C_{Y_1,Y_2}$) is \emph{positively quadrant dependent} (respectively \emph{negatively quadrant dependent}); this means that $C_{Y_1,Y_2}(u, v) \geqslant \Pi(u, v)$ (respectively $C_{Y_1,Y_2}(u, v) \leqslant \Pi(u, v)$) for all $u, v\in[0,1]$. For further explanation of the properties PQD/NQD we refer the reader to \cite[Chapter 5]{Nels} and \cite[p.\ 64]{DuSe}.



An elaboration of \cite[Theorem 3.1(i)]{DuOmOrRu} is clearly obtained using a combination of that result and the obvious fact that reflection is exchanging PQD copulas and NQD copulas.

\begin{proposition}\label{nqd}
  Let $C$ be a PQD copula, or equivalently let $\dot{C}$ be an NQD copula, then $\dot{T_{f,g}}(\dot{C})(u,v)$ is an NQD copula.
\end{proposition}

\textbf{Remark.} We can illustrate this result via Figure 1. Copula $W$ is NQD, so that $\dot{W}=M$ is PQD and Proposition \ref{nqd} gives no information on quadrant dependence in this case. This is in agreement with the third image of Figure 1 on which it is difficult to find any regularity. On the other hand, copula $M$ is PQD, so that $\dot{M}=W$ is NQD and the second image shows a regularity in agreement with the proposition. \qed

We conclude this section with two simple examples which we believe may help us emphasize the importance of the RMM and MM copulas. The first one deals with copulas introduced in \cite{RoLaUbFl}
\begin{equation}\label{RodUbe}
  C(u, v) = uv + f(u)g(v),
\end{equation}
the second one with the famous Eyraud-Farlie-Gumbel-Morgenstern copulas (cf.\ \cite[p.\ 26]{DuSe}), i.e.\ for $\alpha\in[-1,1]$
\begin{equation}%
  C(u, v) = uv + \alpha u(1-u)v(1-v).
\end{equation}

\begin{example}\label{EFGM}
  \begin{description}
     \item[(a)] Let functions $f(t)$ and $g(1-t)$ satisfy conditions {\rm\bf (G1)}, {\rm\bf (G2)} and {\rm\bf (G3)}, and let functions $f$ and $g$ satisfy conditions of \cite[Theorem 2.3]{RoLaUbFl}. Then copula $C$, defined by \eqref{RodUbe} is a maxmin copula.
     \item[(b)] If $\alpha\geqslant0$, then EFGM copula is a maxmin copula, if $\alpha\leqslant0$, then EFGM copula is a reflected maxmin copula.
   \end{description}
\end{example}

\begin{proof}
  Consider $\dot{C}(u,v)=u-C(u,1-v)=uv-f(u)g(1-v)$ to get \textbf{(a)}. Claim \textbf{(b)} follows since $f(u)=\alpha u(1-u)$ and $g(v)=v(1-v)$ satisfy conditions \textbf{(G1)}, \textbf{(G2)}, and \textbf{(G3)}.
\end{proof}

.


\section{ Iterating reflected maxmin copulas }\label{sec:iteration}

In this section we iterate the transformation $\dot{C}\mapsto \dot{T_{\cdot,\cdot}}(\dot{C})$ introduced in the previous section as the authors of \cite{DuOmOrRu} iterated transformation $C\mapsto T_{\cdot,\cdot}(C)$ (in our notation), i.e., we define $\dot{T_{\cdot,\cdot}}^{(n)}:= \dot{T_{\cdot,\cdot}}(\dot{T_{\cdot,\cdot}}^{(n-1)})$. A straightforward application of the mathematical induction principle yields
\begin{equation}\label{inverse_maxmin_n}
    \dot{T_{f,g}}^{(n)}(\dot{C})(u,v) := uv \frac{\dot{C}(\widehat{f}^{(n)}, \widehat{g}^{(n)})} {\widehat{f}^{(n)}\widehat{g}^{(n)}} \prod_{k=0}^{n-1} \max\{0,1- f^*(\widehat{f}^{(k)}(u)) g^*(\widehat{g}^{(k)}(v))\},
\end{equation}
where $\widehat{f}^{(0)}=\widehat{g}^{(0)}=\mathrm{id}$. Since $\widehat{f}(u),\widehat{g}(v)$ are nondecreasing on $[0,1]$ by \cite[Lemma 1, \textbf{(G2)}]{KoOm}, their value is no greater than their value at 1 which is equal to 1, and no smaller than their value at 0 which is equal to 0. So, it follows that the iterates are well-defined. Before we go into showing that all these sequences are convergent, we start by a simple observation based on Proposition \ref{nqd}.

\begin{corollary}\label{pqd}
  If $C$ is a PQD copula, then $\dot{T_{f,g}}^{(n)}(\dot{C}) (u,v)$ is an NQD copula for all $n\in\mathds{N}$.
\end{corollary}

\begin{proof} The result follows inductively by Proposition \ref{nqd}.
\end{proof}

\begin{lemma}\label{first}
  The sequences of functions $\widehat{f}^{(n)},\widehat{g}^{(n)},$ are pointwise convergent. Additionally, if $\alpha,$ respectively $\beta,$ is any of their limit values, then it satisfies the equation $\widehat{f}(\alpha)=\alpha$, respectively $\widehat{g}(\beta)= \beta$.
\end{lemma}

\begin{proof}
We recall the fact that $\widehat{f}(u),\widehat{g}(v),$ are nondecreasing on $[0,1]$ by \cite[Lemma 1, \textbf{(G2)}]{KoOm} implying that all the iterates are monotone nondecreasing. A simple observation $\widehat{f}(u)=u+f(u)\geqslant u$ implies that $\widehat{f}^{(n)}(u)=\widehat{f}^{(n-1)}(\widehat{f}(u))\geqslant \widehat{f}^{(n-1)}(u)$ so that this sequence is monotone for all $u\in [0,1]$; and similarly for $g$. Choose an arbitrary $u\in[0,1]$ and let $\alpha$ be the limit of $\widehat{f}^{(n)}(u)$. Then, $\widehat{f}(\alpha)$ is the limit of $\widehat{f}^{(n+1)}(u)$, so that $\widehat{f}(\alpha)= \alpha$; and similarly for $\beta$.
\end{proof}

\textbf{Remark.} It is an interesting question whether and perhaps when the mapping $\dot{C}\mapsto \dot{T_{f,g}}(\dot{C})$ is a contraction on a compact metric space of all copulas. It is easy to see that for any copulas $C,D$ we have that
\[
    \|\dot{T_{f,g}}(\dot{C})(u,v)-\dot{T_{f,g}}(\dot{D})(u,v)\| = \kappa(u,v) \|\dot{C}(u,v)-\dot{D}(u,v)\|,
\]
where $\kappa(u,v)$ is a product of $u/\widehat{f}(u), v/\widehat{g}(v)$ and $\max\{0,1-f^*(u)g^*(v)\}$. Each of these factors is no greater than 1, however it tends to 1 when $u$ and $v$ tend to 1. So, we have that $0\leqslant\kappa(u,v)\leqslant1$ which proves that this mapping is non-expansive. Since it is not possible to find a constant upper bound $k<1$ for $\kappa$ it is not a contraction. We omit the details not to overlengthen the paper any further. It is somewhat surprising that we are nevertheless able to prove the convergence of the iterates of this map and in the case of Corollary \ref{referee} below even the existence of a unique fixed point to which the iterates converge independently of the starting point. \qed

By \cite[Proposition 2.2.(ii)]{DuOmOrRu} $\widehat{f}$ is equal to the identity function on the interval $[u, 1]$, as soon as we can find an $u\in (0,1]$ such that $\widehat{f}(u) = u$. (Actually, it is useful to write down the fact written there in the original form: if $f(u)=0$ for some $u\in (0,1]$, then $f$ is identically equal to zero on the interval $[u, 1]$.) Let us denote by $\alpha$ the smallest number $u$ with this property. So, we have that $\widehat{f}(u) = u$ for all $u\in[\alpha,1]$ and $\alpha$ is the smallest number of the kind. Observe that $\alpha=0$ is equivalent to saying that $f$ is identically equal to zero on $[0,1]$, and $\alpha=1$ is equivalent to saying that $f$ has no zero on $(0,1)$. The number $\beta$ corresponds to the function $g$ in the same way. So, in particular $\widehat{g}(v) = v$ for all $v\in[\beta,1]$ and $\beta$ is the smallest number with this property. This helps us computing the limits of sequences $\widehat{f}^{(n)}(u),\widehat{g}^{(n)}(v)$:

\begin{lemma}\label{limit}
  The limits of the above sequences exist and are given by
  \begin{align*}
    \widehat{f}^{(\infty)}(u):=\lim_{n\rightarrow\infty}\widehat{f}^{(n)}(u) & = \left\{
    \begin{array}{ll}
        0, & \hbox{$u=0$;} \\
        \alpha, & \hbox{$0<u<\alpha$;} \\
        u, & \hbox{$u\geqslant\alpha$.}
        \end{array}
    \right. \\
    \widehat{g}^{(\infty)}(v):=\lim_{n\rightarrow\infty}\widehat{g}^{(n)}(v) & = \left\{
    \begin{array}{ll}
        0, & \hbox{$v=0$;} \\
        \beta, & \hbox{$0<v<\beta$;} \\
        v, & \hbox{$v\geqslant\beta$.}
        \end{array}
    \right.\\
  \end{align*}
\end{lemma}

We are now in position to compute the limit copula of the sequence given by \eqref{inverse_maxmin_n}. We divide the unit square $[0,1]^2=[0,1]\times [0,1]$ into four rectangles with respect to the lines $u=\alpha$ and $v=\beta$, called suggestively the NW, NE, SW, SE corner of the square. We will show the convergence in passing while computing the limits 
corner by corner. In order to simplify the notation we write
\[
{\overline{C}_{f,g}}^\sigma:=\lim_{n\rightarrow\infty}\dot{T_{f,g}}^{(n)}(\dot{C}).
\]

\begin{lemma}\label{corners}
\begin{description}
  \item[(a)] For $u\geqslant\alpha$ and $v\geqslant\beta$ we have
    \[
        {\overline{C}_{f,g}}^\sigma(u,v)=\dot{C}(u,v).
    \]
  \item[(b)] For $u\geqslant\alpha$ and $v\leqslant\beta$ we have
    \[
        {\overline{C}_{f,g}}^\sigma(u,v)=\frac{v}{\beta}\dot{C}(u,\beta).
    \]
  \item[(c)] For $u\leqslant\alpha$ and $v\geqslant\beta$ we have
    \[
        {\overline{C}_{f,g}}^\sigma(u,v)=\frac{u}{\alpha}\dot{C}(\alpha,v).
    \]
  \item[(d)] In the remaining case it holds that
\[
    {\overline{\Pi}_{f,g}}^\sigma(u,v)=\left\{
                                             \begin{array}{ll}
                                               {\displaystyle uv\prod_{k=0}^{\infty} \left(1 - f^*(\widehat{f}^{(k)}(u))g^*(\widehat{g}^{(k)}(v))\right)}
, & \hbox{if $f(u)g(v)<uv$;} \\
                                               0, & \hbox{if $f(u)g(v)\geqslant uv$;}
                                             \end{array}
                                           \right.
\]
where the product in the first case is always convergent.
\end{description}
  \end{lemma}

\begin{proof} \textbf{(a)}
  Since $\widehat{f}^{(\infty)}(u)=u$ and $\widehat{g}^{(\infty)}(v)=v$, it remains to see that the product in \eqref{inverse_maxmin_n} equals to 1. Indeed, $\widehat{f}(u)=u$, so that $\widehat{f}^{(k)}(u)=u$ for all $k\geqslant0$, while $f(u)=0$ and consequently $f^*(u)=0$, and the same kind of considerations apply to $g$. So, it follows that the product on the right hand side of \eqref{inverse_maxmin_n} is identically equal to one. The proof of \textbf{(b)} and \textbf{(c)} is based on similar ideas.

  \textbf{(d)} We first assume that $\alpha=\beta=1$. If $f(u)g(v) \geqslant uv$, the first factor of the product in \eqref{inverse_maxmin_n} is zero and the desired conclusion follows. If $f(u)g(v) < uv$, we only need to show convergence of the product, the rest is obvious from \eqref{inverse_maxmin_n}. Now, the first factor, and consequently all of them, are strictly positive and strictly smaller than $1$. So, the sequence of partial products is positive and decreasing and always converges. However, an infinite product is convergent only if in addition its limit is nonzero. In case that $u+v>1$ we use the Fr\'{e}chet-Hoeffding lower bound for our copula to see that this is indeed so and the desired result follows. It remains to treat the general case of $u,v>0$ which we will reduce to the previous one. We will use a well-known fact that (C) \emph{a product $\prod_{k=0}^{\infty}a_k$ converges if and only if the series $\sum_{k=0}^{\infty}(1-a_k)$ converges. }This proves, in particular, that the series
\begin{equation}\label{series}
  \sum_{k=0}^{\infty} f^*(\widehat{f}^{(k)}(u))g^*(\widehat{g}^{(k)}(v))
\end{equation}
converges for $u+v>1$ by the above. 
Next, if $u,v>0$, then the sequences $\widehat{f}^{(k)} (u)$ respectively $\widehat{g}^{(k)}(v)$ are nondecreasing and converging to $\alpha=1$ respectively $\beta=1$. So, given $j$ big enough we have $\widehat{f}^{(k)} (u)+\widehat{g}^{(k)}(v)>1$ for all $k \geqslant j$. So, if we denote $u_0=\widehat{f}^{(j)} (u)$ and $v_0=\widehat{g}^{(j)}(v)$ we get
\[
    \sum_{k=j}^{\infty} f^*(\widehat{f}^{(k-j)}(u_0)) g^*(\widehat{g}^{(k-j)}(v_0)) = \sum_{k=0}^{\infty} f^*(\widehat{f}^{(k)}(u)) g^*(\widehat{g}^{(k)}(v)),
\]
so that the series on the left hand side must be convergent because it is a residual series of a convergent series and $u_0+v_0>1$. It follows that the series on the right hand side is convergent and consequently the product is convergent by (C).

It remains to treat the case that either $\alpha<1$ or $\beta<1$. If we choose $u=\alpha$ and $v=\beta$, we get easily ${\overline{\Pi}_{f,g}}^\sigma(u,v)=\alpha\beta>0$. Now, these iterates are copulas and so is the limit, yielding that it is continuous. So, for $u$ close to $\alpha$ and $v$ close to $\beta$, we still have, say, ${\overline{\Pi}_{f,g}}^\sigma(u, v) \geqslant \alpha \beta /2 > 0$. Since the limit is positive, the rest of the proof goes as above.
\end{proof}

\begin{theorem}\label{thm:main}
  The limit of the sequence \eqref{inverse_maxmin_n} always exists and is equal to $\dot{C}(u,v)$ on the NE corner, to
\[
    \frac{v}{\beta}\dot{C}(u,\beta),\quad\mbox{respectively}\quad
    \frac{u}{\alpha}\dot{C}(\alpha,v),
\]
on the NW, respectively the SE corner; while on the SW corner it equals to either
\[
    uv\frac{\dot{C}(\alpha, \beta)} {\alpha\beta} \prod_{k=0}^{\infty} \left(1 - f^*(\widehat{f}^{(k)}(u))g^*(\widehat{g}^{(k)}(v))\right),
\]
if  $f(u)g(v)<uv$ (in this case the product always converges), or zero if  $f(u)g(v)\geqslant uv$.
\end{theorem}

\begin{proof}
  The form of the limit on the NW, NE, respectively SE corners is given by Lemma \ref{corners}\textbf{(a)}, \textbf{(b)} respectively \textbf{(c)}. To get the desired result on the SW corner, write the right-hand side of \eqref{inverse_maxmin_n} as
\[
\frac{\dot{C}(\widehat{f}^{(n)}(u), \widehat{g}^{(n)}(v))} {\widehat{f}^{(n)}(u)\widehat{g}^{(n)}(v)} \dot{T_{f,g}}^{(n)}(\dot{\Pi})(u,v)
\]
and the theorem follows by Lemma \ref{limit} and Lemma \ref{corners}\textbf{(d)}.
\end{proof}

\begin{corollary}\label{referee}
   Assume that $\alpha=\beta=1$. Then
$${\overline{C}_{f,g}}^\sigma={\overline{\Pi}_{f,g}}^\sigma$$
for any copula $C$.
\end{corollary}

The scatterplots presented on Figure 2 
are giving us some feeling of the convergence of the iteration we are studying in this section. Each image in the first row shows the scatterplot of the copula obtained when transformation $\dot{T}^{(1)}$ is applied to a starting copula $C$ and the image below shows the one obtained when $\dot{T}^{(2)}$ is applied to the same starting copula. The starting copulas are respectively $\Pi,M$, and $W$. The choice of generating functions in all these examples are $f(u)= \displaystyle \frac{1}{2}(1-u)$ and $g(v)=\displaystyle \frac{1}{10}(1-v)$. In all the cases the dots are scattered above a hyperbola. In the first column only the density of the dots varies with iteration. In the second column a singular line appeares that converges to the top of the square with iterations. In the third column a singular line appears that cuts out a piece of the area only on the plot of $\dot{T}^{(1)}$, while later the whole area above the hyperbola is dotted.

\begin{figure}[h!]\label{fig:slika2}
            \includegraphics[width=0.32\textwidth]{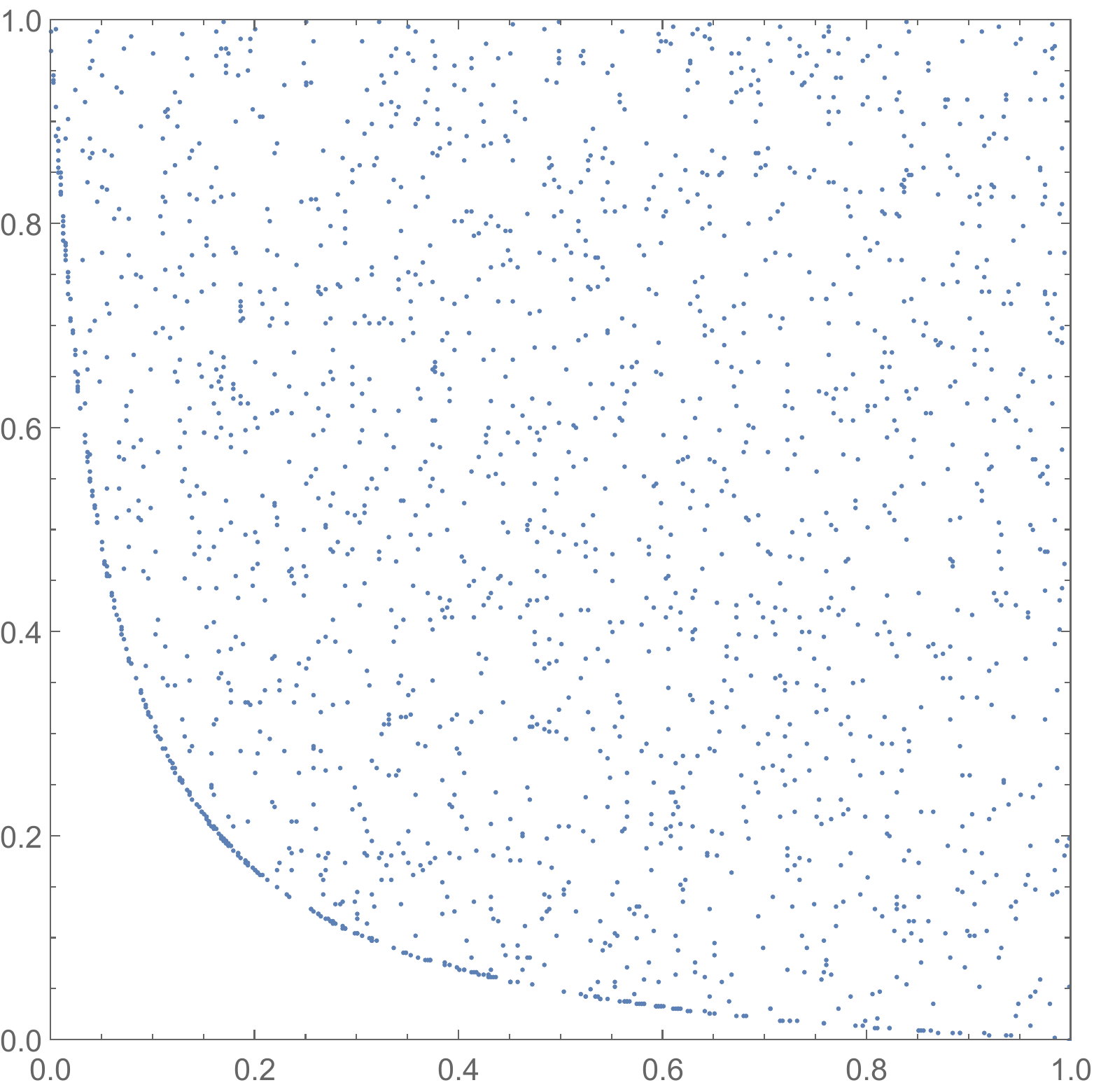} \hfil \includegraphics[width=0.32\textwidth]{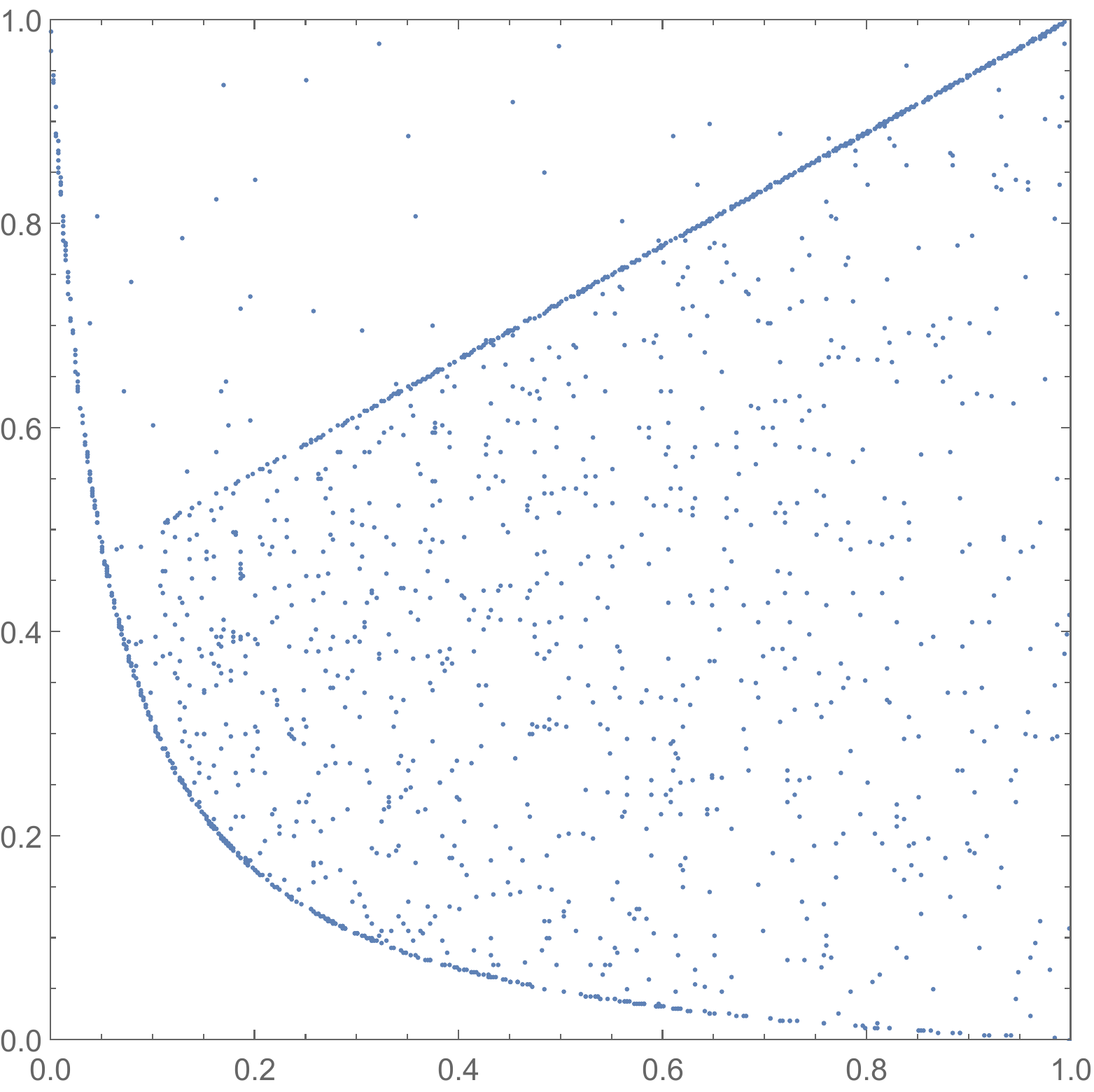} \hfil \includegraphics[width=0.32\textwidth]{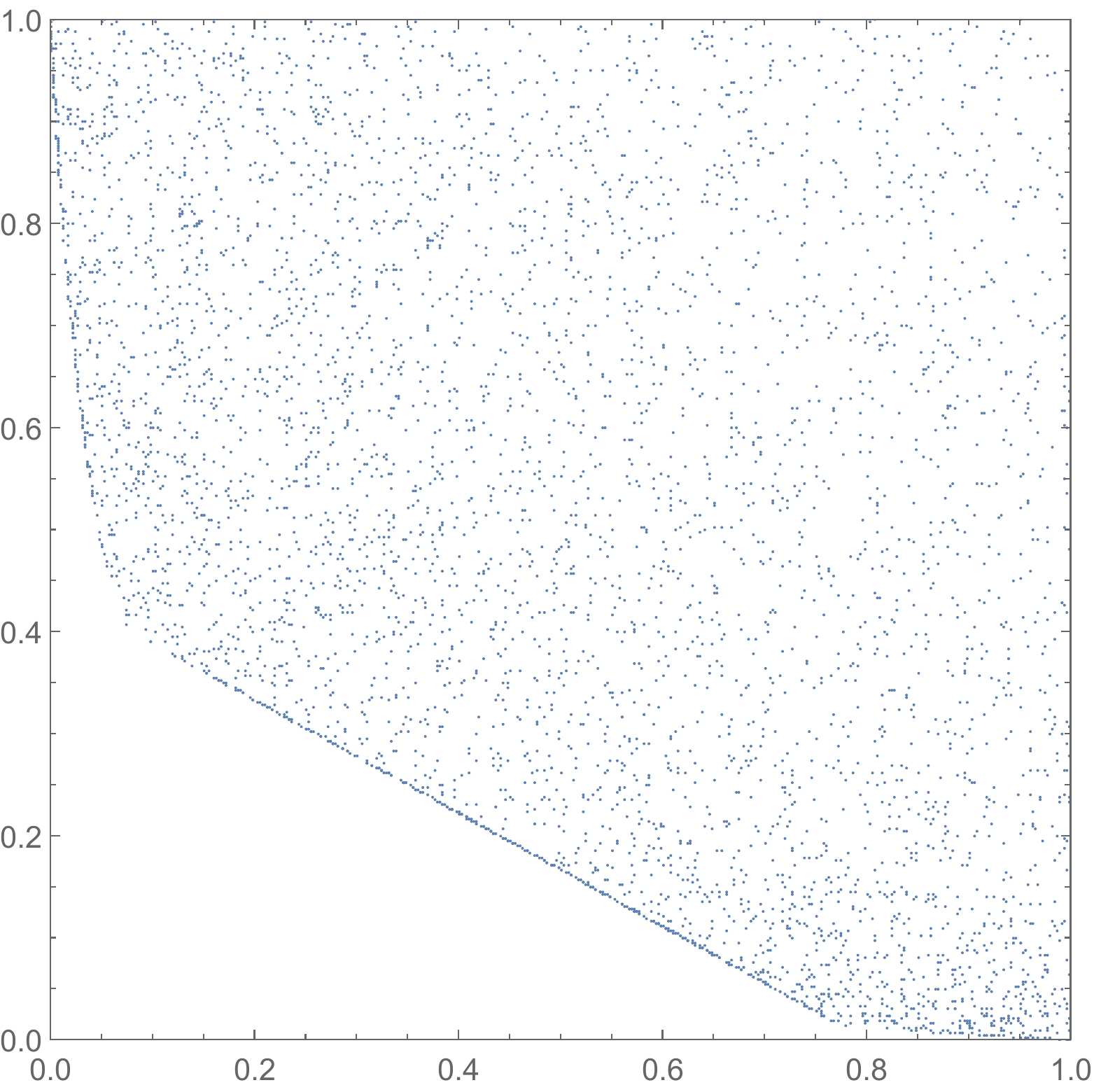}   \includegraphics[width=0.32\textwidth]{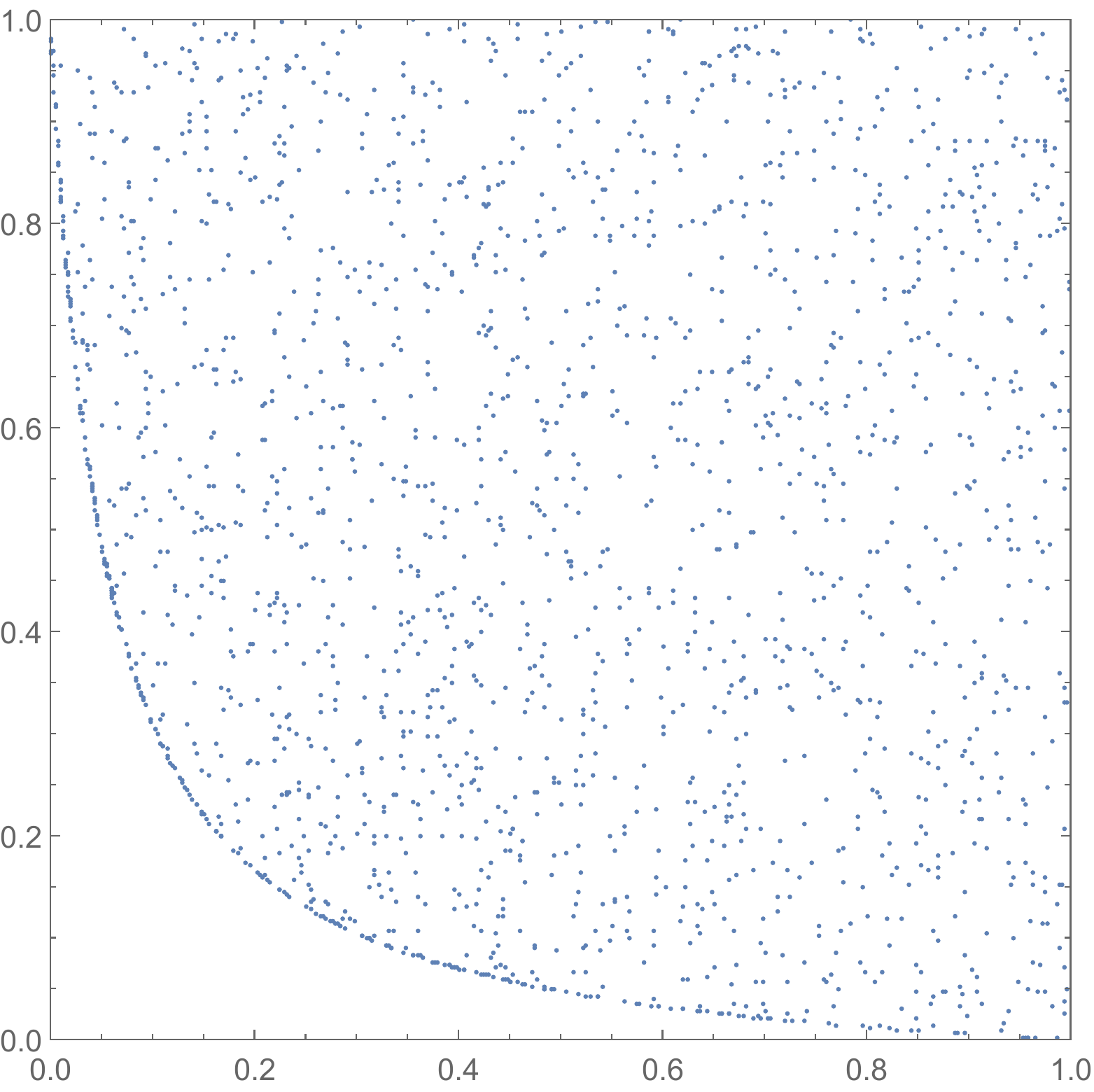} \hfil \includegraphics[width=0.32\textwidth]{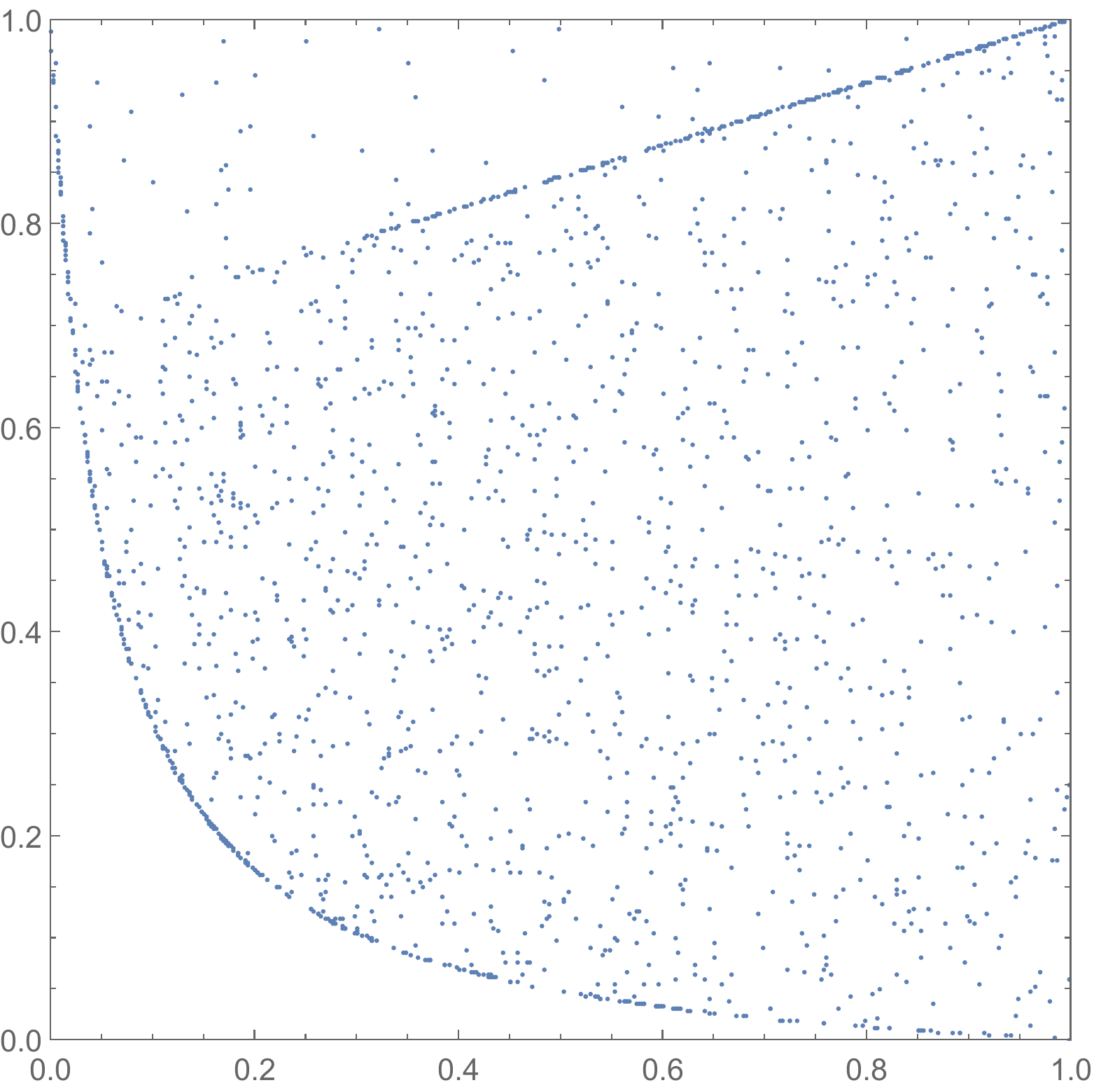} \hfil  \includegraphics[width=0.32\textwidth]{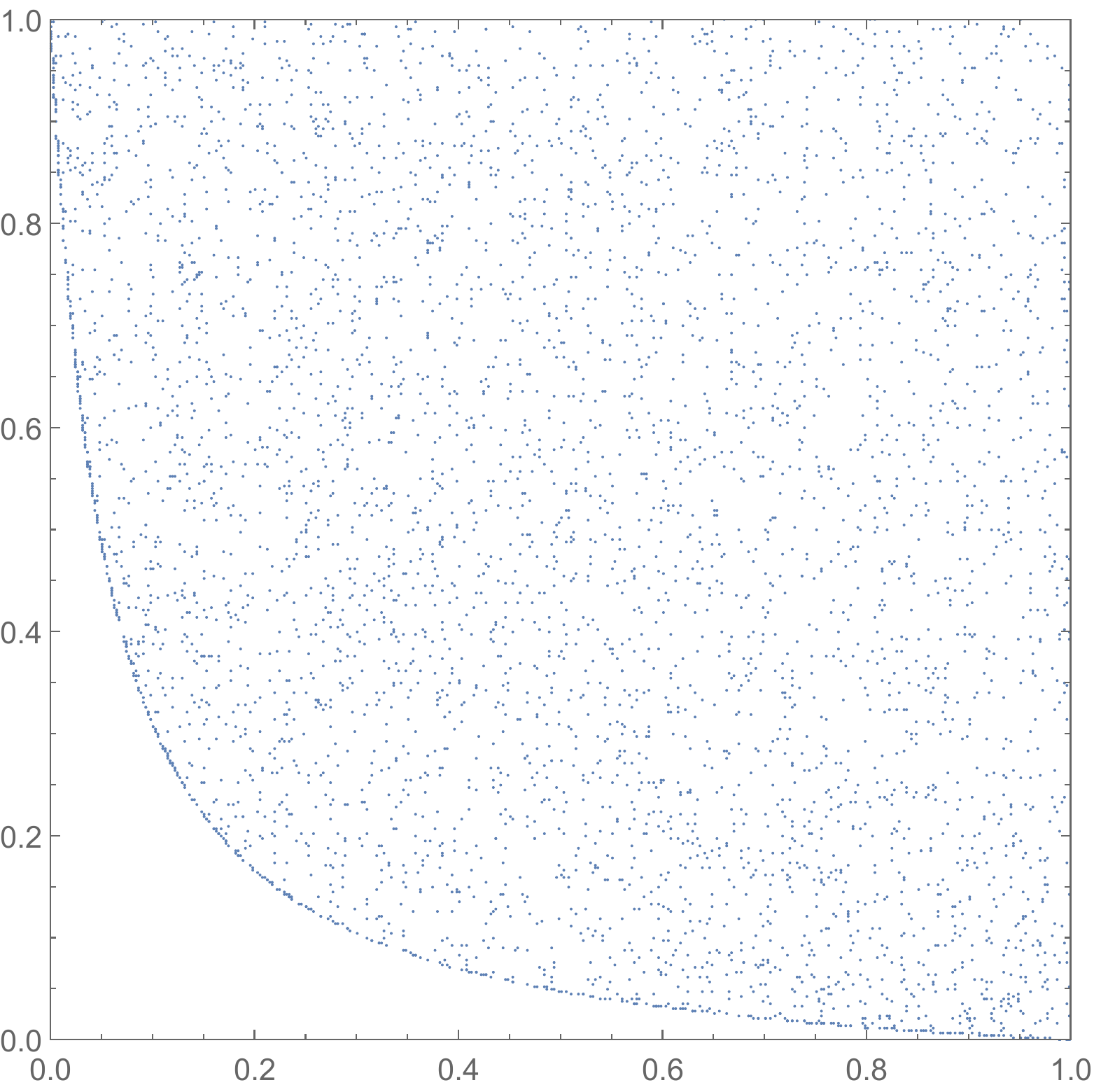}
            \caption{ Transformation $\dot{T}^{(1)}$ and $\dot{T}^{(2)}$ on copulas $\dot{\Pi},\dot{M}$, and $\dot{W}$, respectively. }
\end{figure}

\begin{figure}[h!]\label{fig:slika2}
            \includegraphics[width=0.32\textwidth]{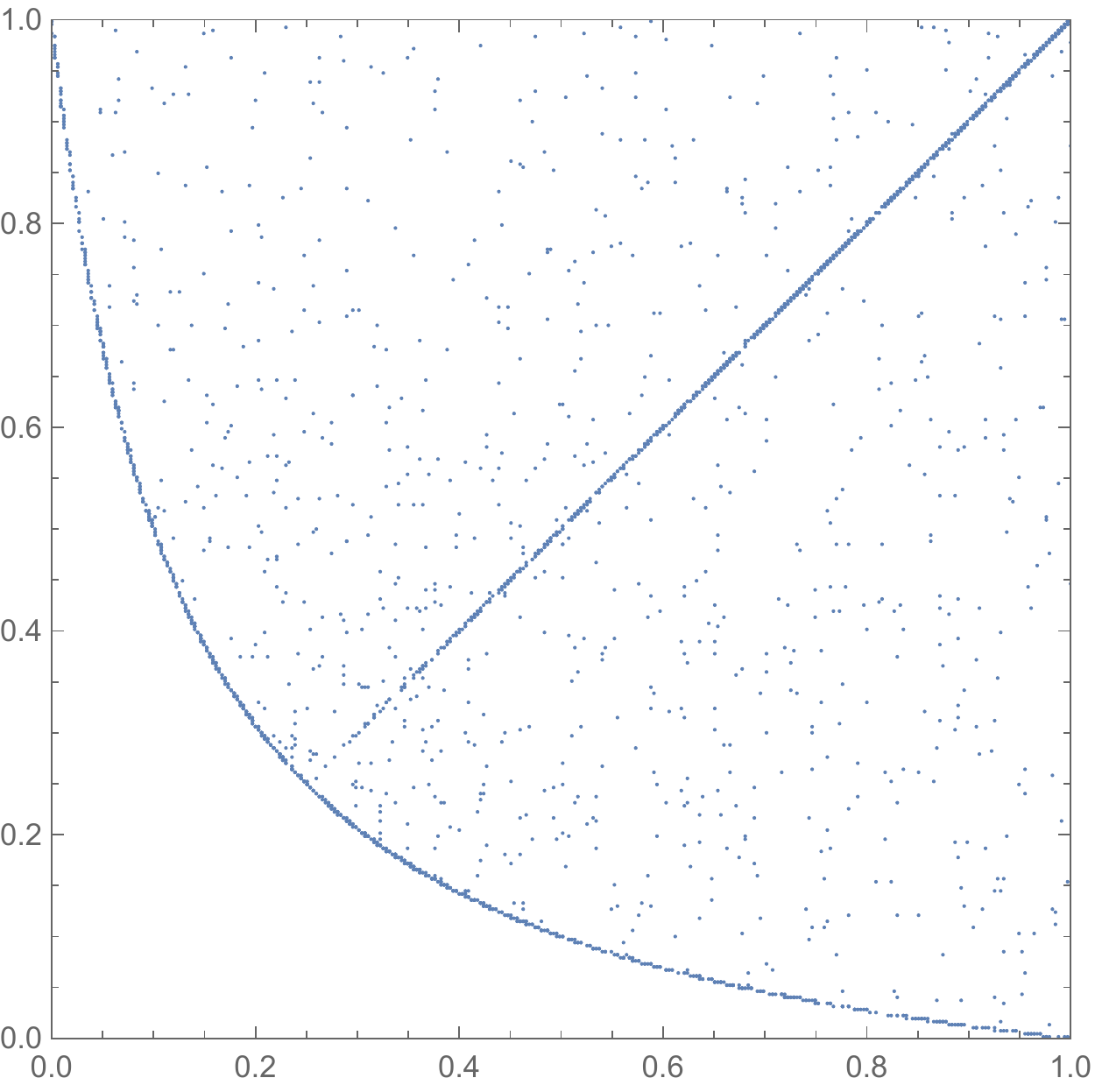} \hfil \includegraphics[width=0.32\textwidth]{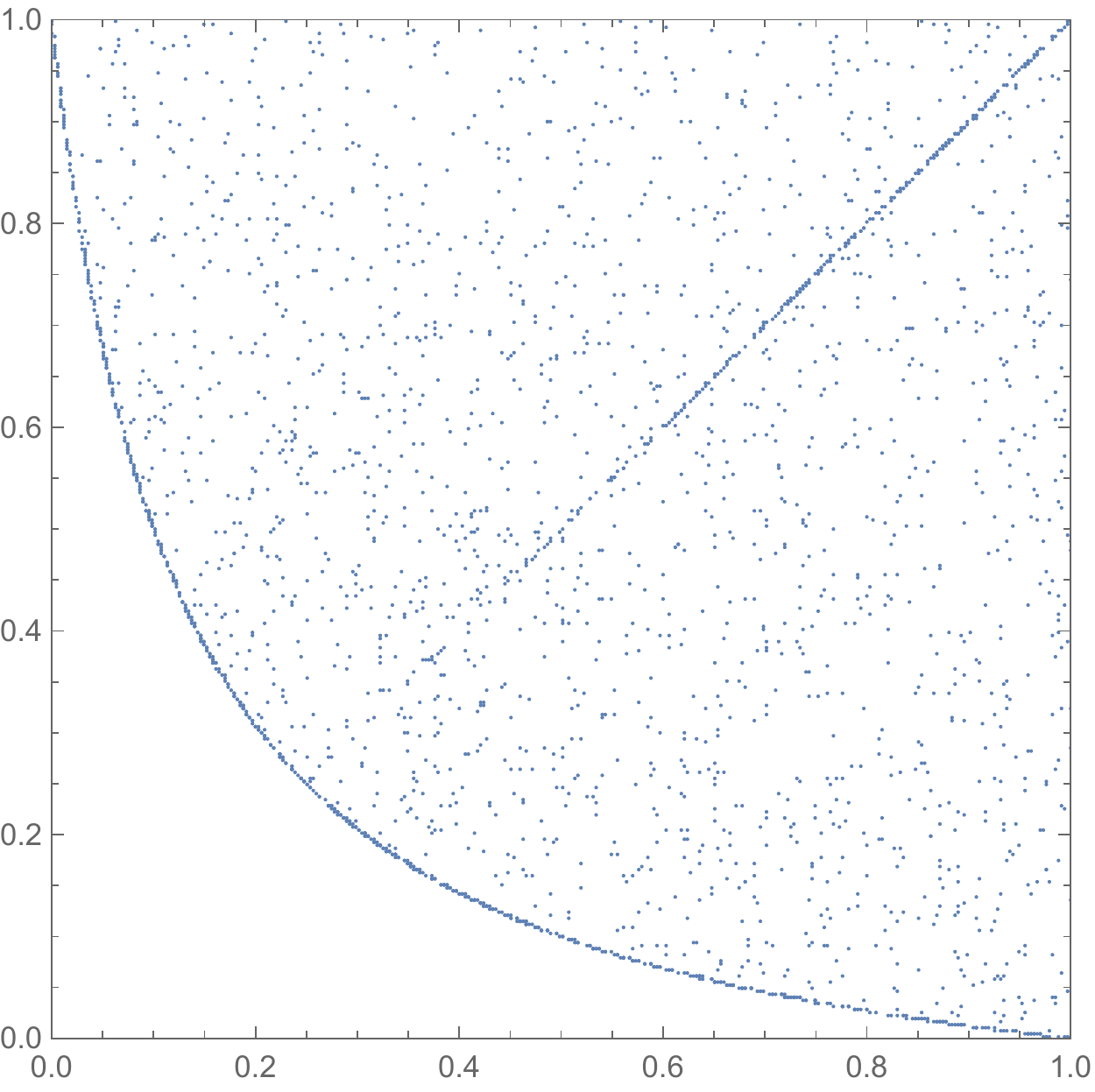} \hfil \includegraphics[width=0.32\textwidth]{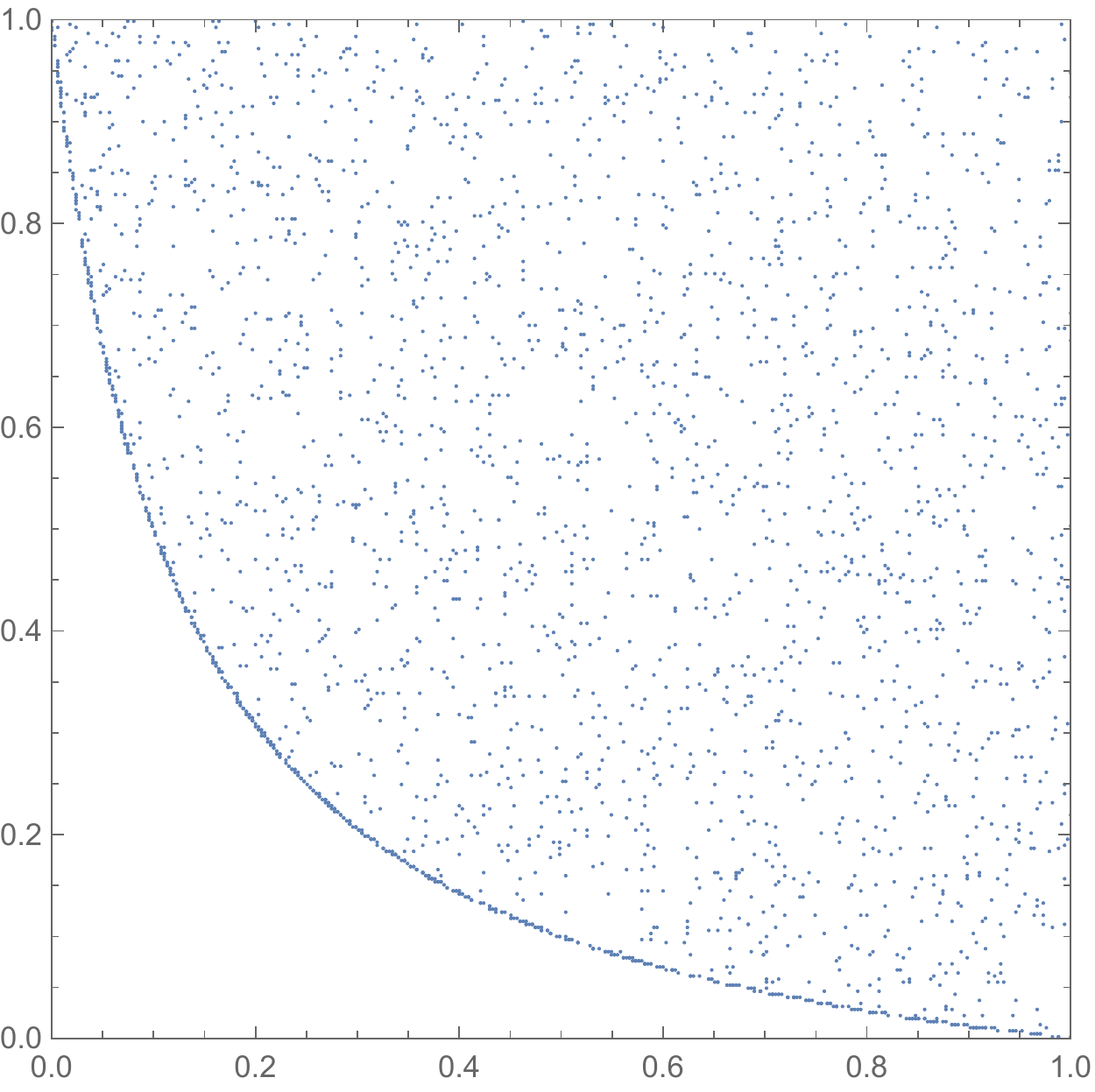} \\ \includegraphics[width=0.32\textwidth]{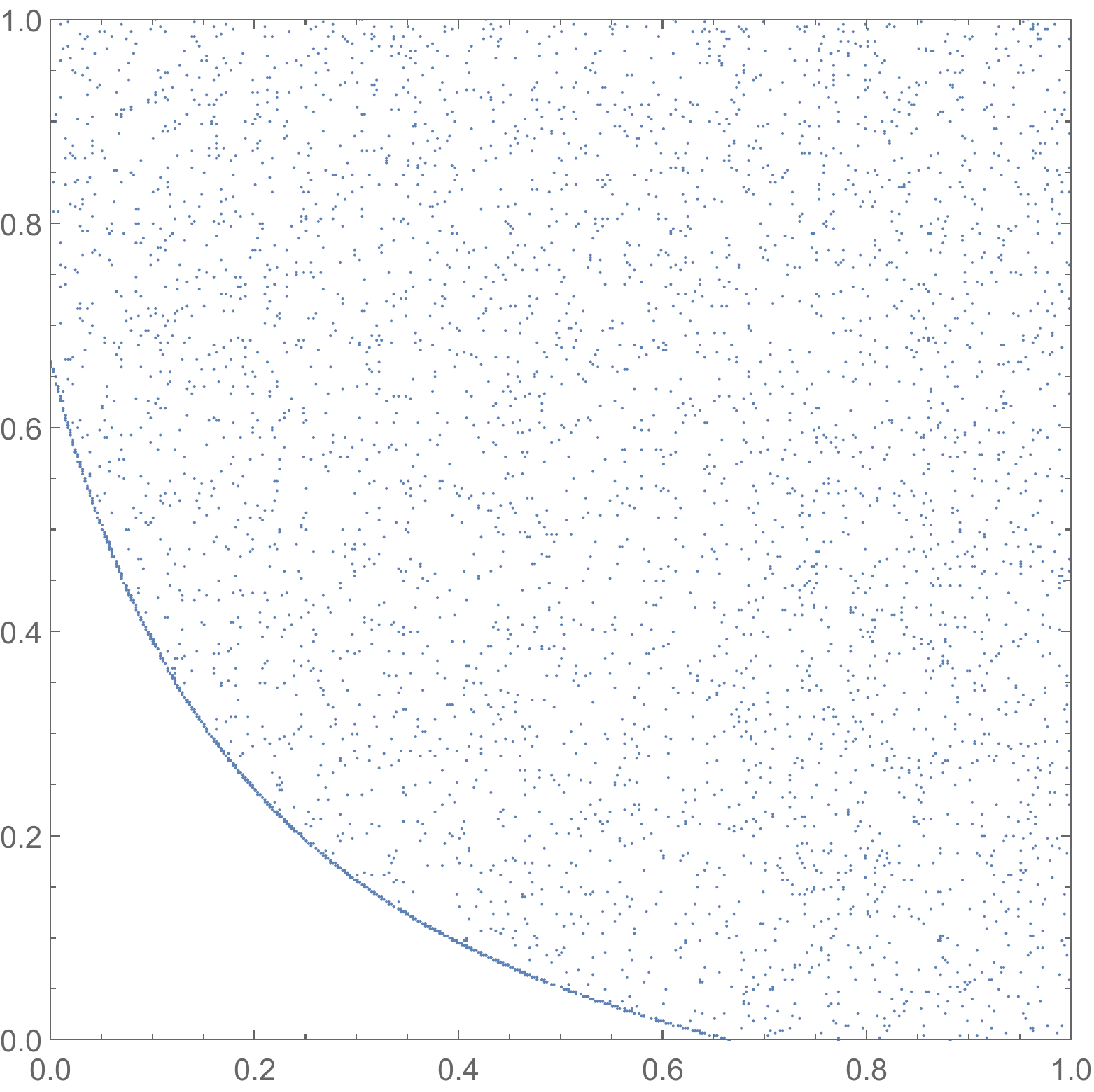} \hfil \includegraphics[width=0.32\textwidth]{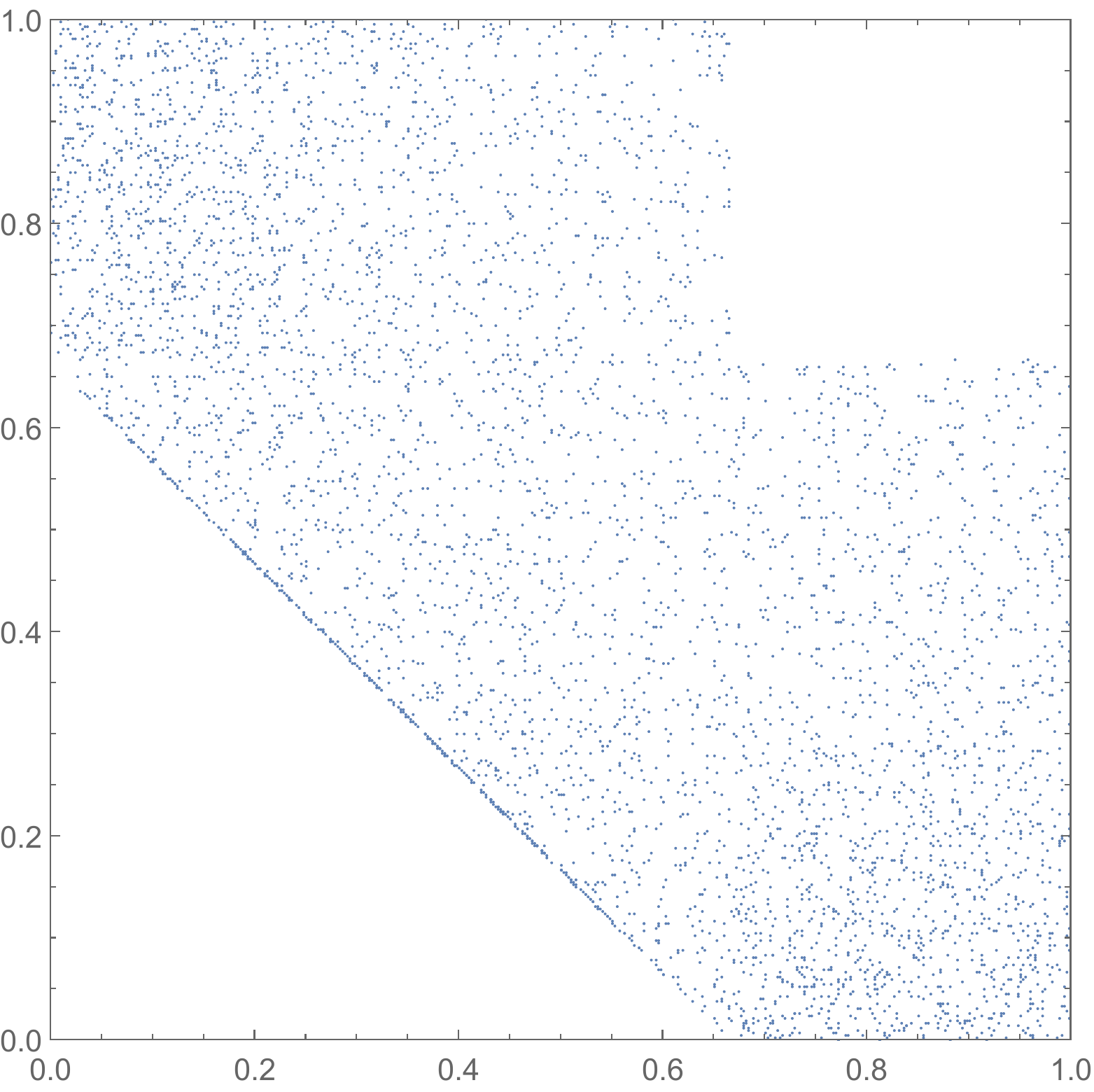} \hfil  \includegraphics[width=0.32\textwidth]{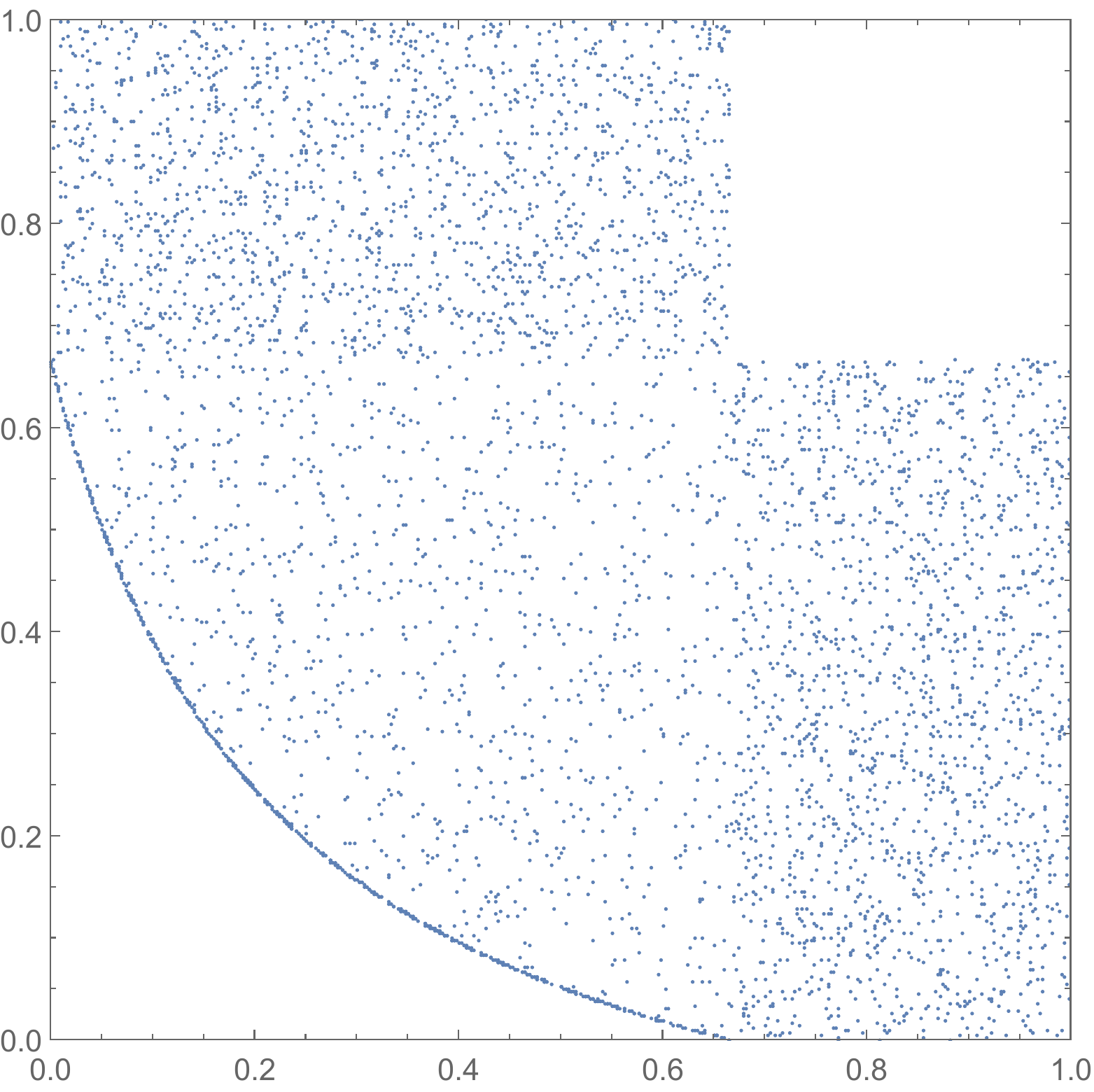} \\ \hfil
\includegraphics[width=0.32\textwidth]{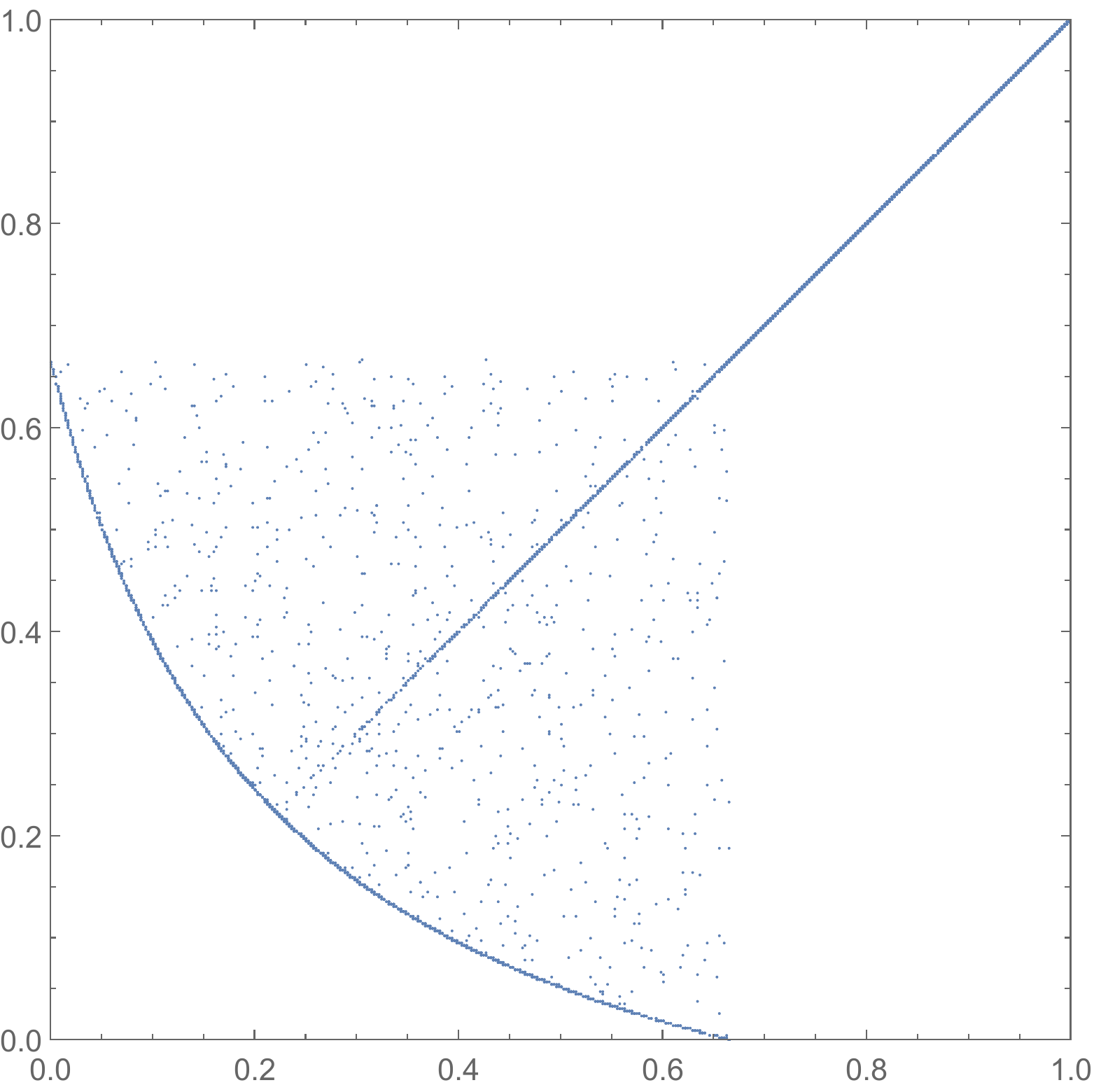} \hfil \includegraphics[width=0.32\textwidth]{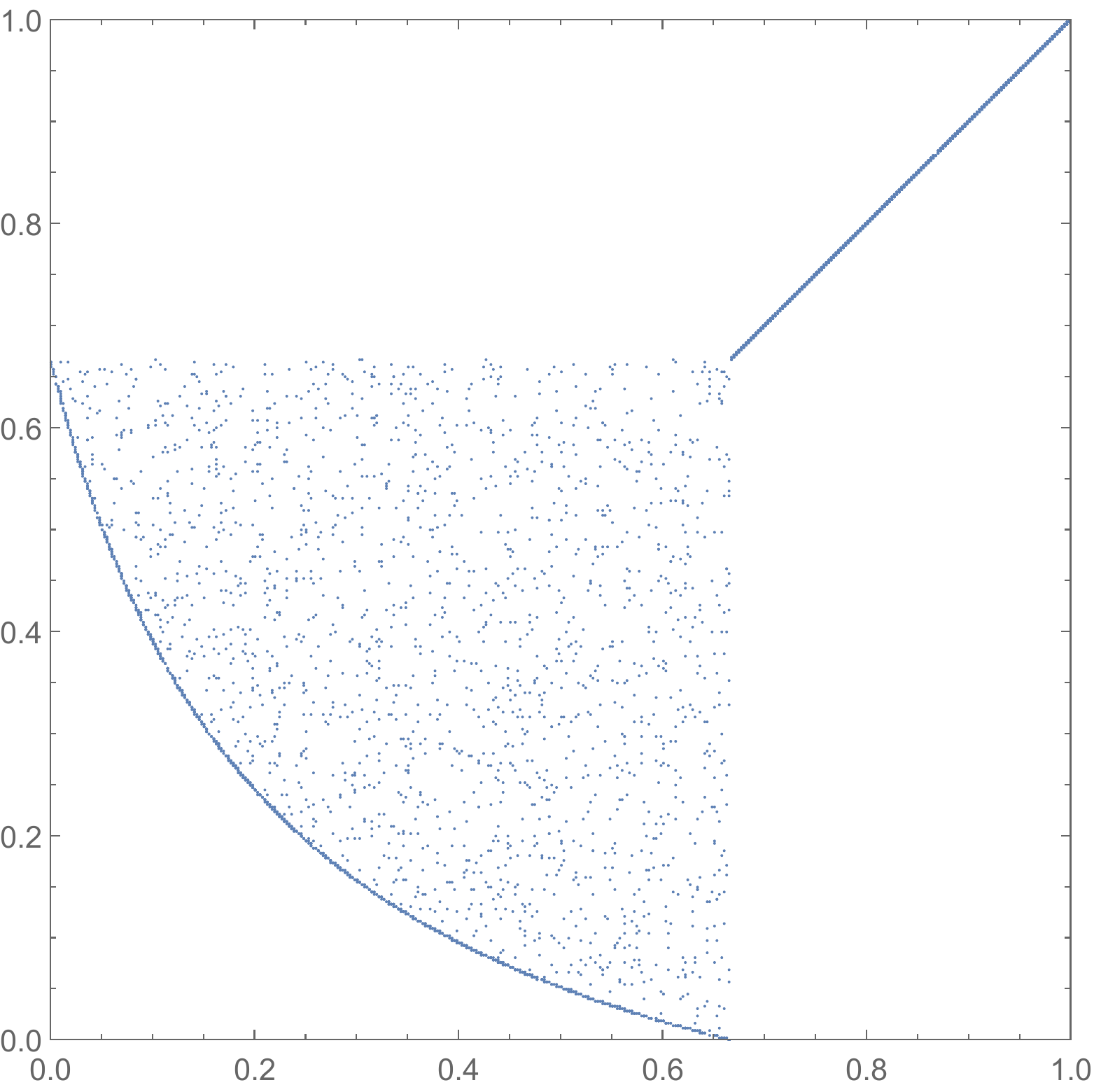} \hfil
            \caption{ Various transformations $\dot{T}^{(n)}$ on copulas $\dot{\Pi},\dot{M}$, and $\dot{W}$ as described in the text. }
\end{figure}

In Figure 3 we present some more interesting scatterplots. We give here the description of the images row by row going from the top to the bottom and within the row from the left to the right. Generating functions for the first three images are $\displaystyle f(u)=\frac{1}{3} (1-u)$ and $\displaystyle g(v)=\frac{1}{3} (1-v)$, while the copula equals $C=W$, and respectively $n=1$, $n=3$, $n=\infty$. For the rest of the images the first generating function is $\displaystyle f(u)=\frac{1}{2} \max \{0, \frac{2}{3}-u \}$ and $g(v)$ is as in the previous case. The first image in the second row corresponds to the copula $C=\Pi$ and $n=1$. For larger values of $n$ the image is very similar in this case. The next two images correspond to the case $C=M$, and $n=1$, respectively $n=\infty$. The last two images correspond to $C=W$, and $n=1$, respectively $n=\infty$.

\begin{figure}[h!]\label{fig:slika2}
            \includegraphics[width=0.32\textwidth]{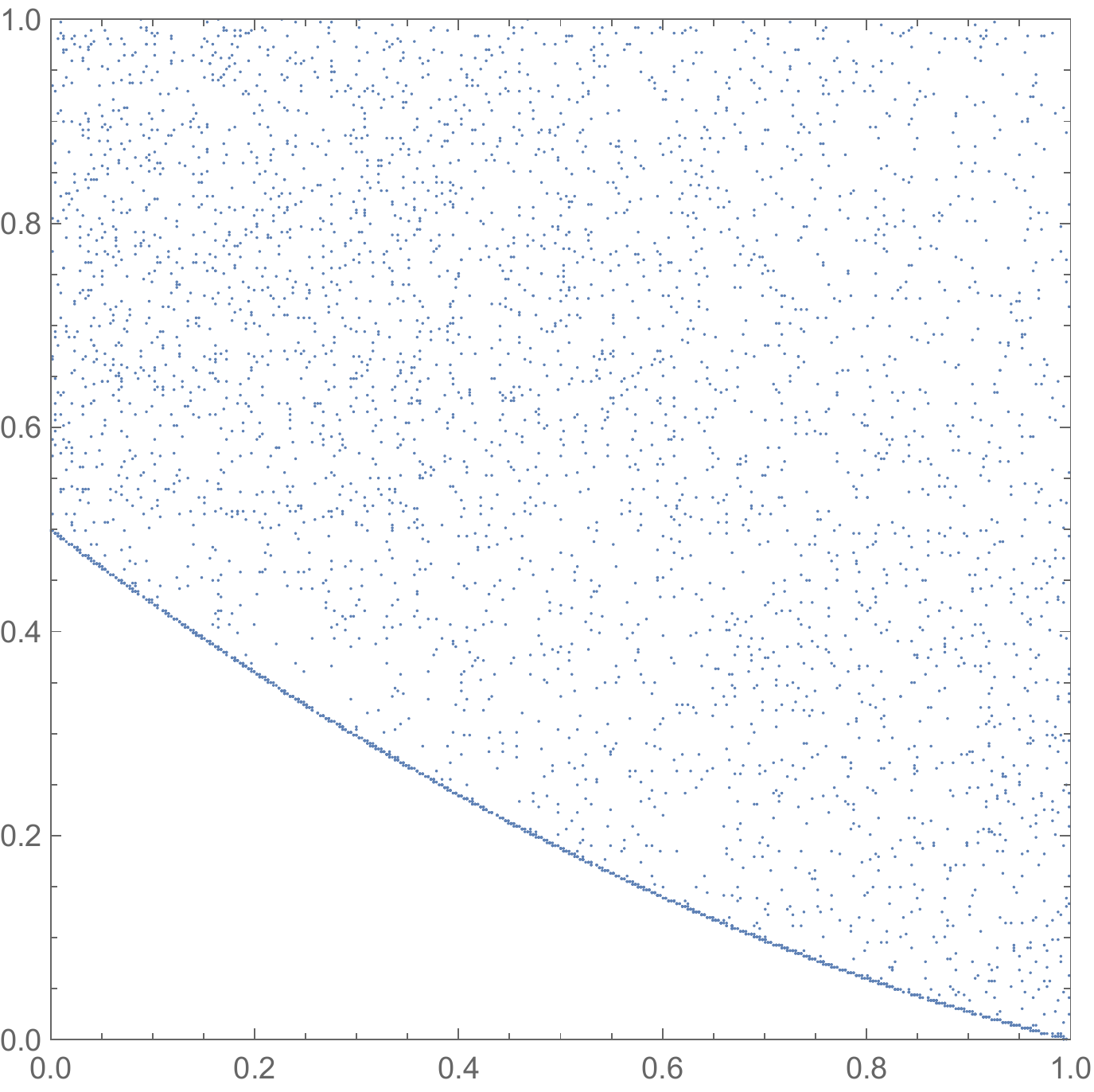} \hfil \includegraphics[width=0.32\textwidth]{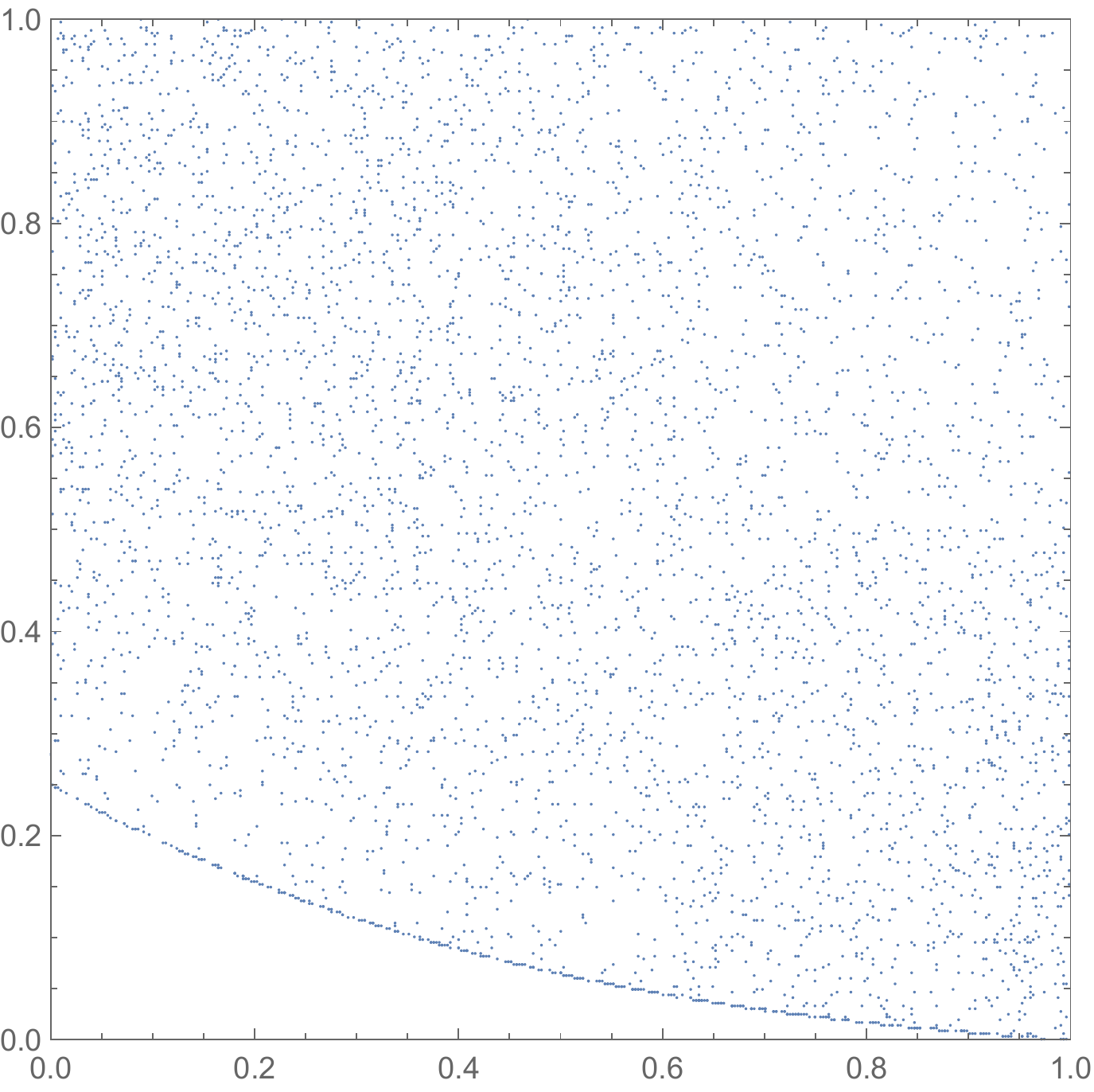} \hfil \includegraphics[width=0.32\textwidth]{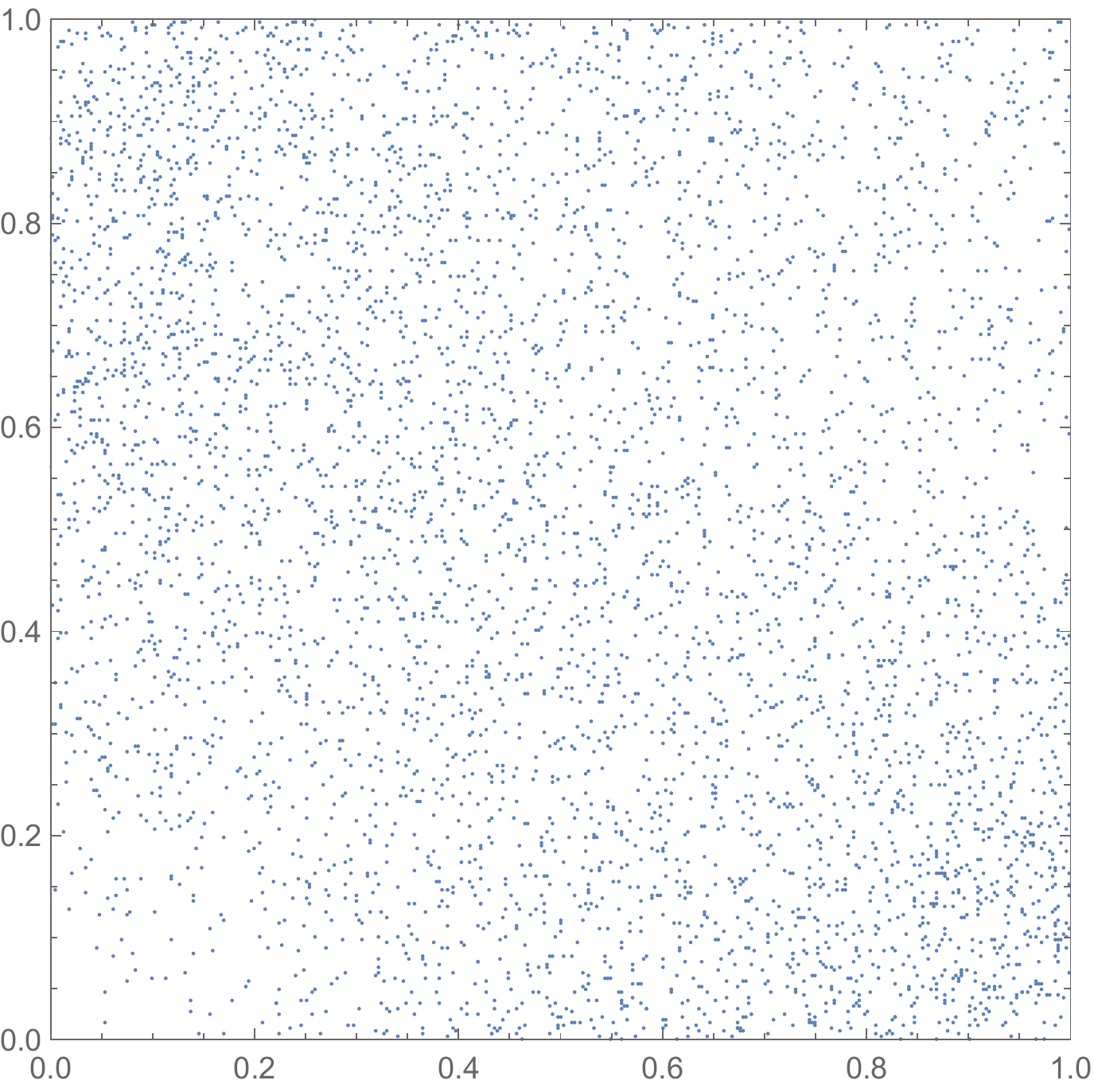}
            \caption{ Evolution of the singular part. }
\end{figure}

On Figure 4 we consider evolution of the singular part for the example of an RMM copula given by $f(u)=\min\{u,1-u\}, g(v)=v(1-v)$, and $C=M$. The resulting copulas for $n=1,2$ are shown on the first two scatterplots of Figure 4 and are clearly singular. When $n$ approaches infinity, the singularity is pushed to the lower edge of the square and in the limit it disappears. This is seen on the third scatterplot of this figure. The limiting copula is equal to ${\overline{\Pi}_{f,g}}^\sigma$ as explained in Corollary \ref{referee}.

Let us conclude this section by translating Theorem \ref{thm:main} in terms of maxmin copulas. The transformation $\dot{C}\mapsto \dot{T_{f,g}}(\dot{C})$ is obtained from the transformation $C\mapsto T_{\phi,\psi}(C)$ by applying the flip on either side of the transformation. Since the flip is an involutory operation (cf.\ \cite[p.\ 33]{DuSe}), it follows inductively that the transformation $\dot{C} \mapsto \dot{T_{f,g}}^{(n)} (\dot{C})$ is obtained from the transformation $C\mapsto {T_{\phi,\psi}}^{(n)}(C)$ (here we adjust the notation of \cite{DuOmOrRu} in an obvious way) by applying the flip on both sides of the transformation. So, we can compute the formula for ${T_{\phi, \psi}}^{(n)}(C)$ from the formula for $\dot{T_{f,g}}^{(n)}(\dot{C})$ by \textbf{(a)} applying the flip on both sides of $\eqref{inverse_maxmin_n}$ and \textbf{(b)} replacing the generators $f,g$ and their auxiliary functions by the generators $\phi,\psi$ and their auxiliaries using relations \eqref{inverse_generators}. When replacing the iterates a simple inductive argument reveals
\begin{equation}\label{eq:orig_generators}
  \widehat{f}^{(n)}(u)=\phi^{(n)}(u)\ \ \mbox{and}\ \ \widehat{g}^{(n)}(v)=1 - \psi^{(n)}(1 - v),
\end{equation}
which implies
  \begin{align*}
    \phi^{(\infty)}(u):=\lim_{n\rightarrow\infty}\phi^{(n)}(u) & = \left\{
    \begin{array}{ll}
        0, & \hbox{$u=0$;} \\
        \alpha, & \hbox{$0<u<\alpha$;} \\
        u, & \hbox{$u\geqslant\alpha$.}
        \end{array}
    \right. \\
    \psi^{(\infty)}(v):=\lim_{n\rightarrow\infty}\psi^{(n)}(v) & = \left\{
    \begin{array}{ll}
        v, & \hbox{$v\leqslant1-\beta$;} \\
        1-\beta, & \hbox{$1-\beta<v<1$;} \\
        1, & \hbox{$v=1$.}
        \end{array}
    \right.\\
  \end{align*}

Here, $\alpha$ is the smallest $u\in[0,1]$ such that $\phi$ is equal to identity function on $[u,1]$, and $1-\beta$ is the greatest $v\in[0,1]$ such that $\psi$ is equal to identity function on $[0,v]$ (cf.\ \cite[Proposition 2.2]{DuOmOrRu}). We divide again the square $[0,1]^2$ into four corners: The limit function in $u$ is equal to a constant for $u\leqslant\alpha$ and to the identity otherwise; the limit function in $v$ is equal to the identity for $v\leqslant1-\beta$ and to a constant otherwise.

In order to transform equation \eqref{inverse_maxmin_n}, we need to replace functions $f^*,g^*$ using the functions $\phi^*,\psi_*$ introduced in \cite[p.\ 156]{DuOmOrRu}. It is not hard to find the relation between $f^*$ and $\phi^*$, and between $g^*$ and $\psi_*$. We need formulas \[
\begin{split}
   \phi^*(u) &=\frac{u}{\widehat{f}(u)} = \frac{1}{f^*(u)+1},\ \ \mbox{and} \\
   \psi_*(1-v)  &  =\frac{1-v-\psi(1-v)}{1-\psi(1-v)}= \frac{g(v)}{\widehat{g}(v)}=\frac{g^*(v)}{g^*(v)+1},
\end{split}
\]
so that
\begin{equation}\label{orig_stars}
  f^*(u)=\frac{1-\phi^*(u)}{\phi^*(u)}\ \ \mbox{and}\ \ g^*(v)= \frac{\psi_*(1-v)}{1-\psi_*(1-v)}
\end{equation}

Using formulas \eqref{eq:orig_generators} and \eqref{orig_stars} we can translate equation \eqref{inverse_maxmin_n} into:

\begin{align}\label{maxmin_n}
\begin{split}
   {T_{\phi,\psi}}^{(n)}({C})(u,v) := u- u(1-v) \frac{\phi^{(n)}(u)- {C}(\phi^{(n)}(u), \psi^{(n)}(v))} {\phi^{(n)}(u)(1-\psi^{(n)}(v))}\times \\
   \prod_{k=0}^{n-1} \max\left\{0,1-\frac{1- \phi^*(\phi^{(k)}(u))} {\phi^*(\phi^{(k)}(u))} \frac{\psi_*(\psi^{(k)}(v))} {1-\psi_*(\psi^{(k)}(v))}\right\}.
\end{split}
\end{align}

\begin{theorem}\label{thm:maxmin}
  The limit of the  sequence \eqref{maxmin_n} always exists and equals
\begin{description}
  \setlength\itemsep{.5em}
  \item[(a)] ${C}(u,v)$ on the SE corner;
  \item[(b)] $\displaystyle \frac{u}{\alpha}{C}(\alpha,v)$ on the SW corner;
  \item[(c)] $\displaystyle u-\frac{1-v}{\beta}(u-{C}(u,1-\beta)),$ on the NE corner;
  \item[(d)] either
      \[
          u- \frac{u(1-v)}{\alpha\beta}(\alpha-{C}(\alpha,1-\beta))
          \prod_{k=0}^{\infty} \left(1 - \frac{1- \phi^*(\phi^{(k)}(u))} {\phi^*(\phi^{(k)}(u))} \frac{\psi_*(\psi^{(k)}(v))} {1-\psi_*(\psi^{(k)}(v))}\right)
      \]
        if $\phi^*(u)>\psi_*(v)$ (in this case the product always converges), or $u$ if $\phi^*(u)\leqslant\psi_*(v)$, on the NW corner.
\end{description}
\end{theorem}

\begin{proof} Straightforward computation.
\end{proof}

\section{ Inheritance of dependence properties via reflected maxmin transformation }\label{sec:properties}

Using Theorem \ref{thm:main} we upgrade the result presented in Corollary \ref{pqd}. Here and in the rest of this section we denote by $\overline{C}_{\phi,\psi}$ the limit copula of the sequence ${T_{\phi,\psi}}^{(n)}({C})$.

\begin{theorem}
  \begin{description}
    \item[(a)] If copula $C$ is PQD, or equivalently $\dot{C}$ is NQD, then ${\overline{C}_{f,g}}^\sigma$ is NQD.
    \item[(b)] If either the function $f(x)$ is nonzero for all $x\in(0,1)$ or the function $g(x)$ is nonzero for all $x\in(0,1)$, then ${\overline{C}_{f,g}}^\sigma$ is an NQD copula for any copula $C$.
    \item[(c)] If copula $C$ is PQD, then $\overline{C}_{\phi,\psi}$ is PQD.
    \item[(d)] If either for the function $\phi$ it holds that $\phi(x) \ne x$ for all $x \in (0, 1)$ or for the function $\psi$ it holds that $\psi(x) \ne x$ for all $x \in (0, 1)$, then $\overline{C}_{\phi,\psi}$ is a PQD copula for any copula $C$.
  \end{description}
\end{theorem}

We will omit the proof since it goes in an obvious way.\\

{
\begin{table}
  \caption{Spearman's rho for copulas $\Pi,M$, $W$ and Clayton copula $K_{-0.7}$}
  \centering

  \begin{tabular}{|rrrrrrrr|}
\hline\hline
$\alpha$ & $\beta$ & \qquad & $\rho(\dot{\Pi})$ & $\rho(\dot{T}^{(1)}\dot{\Pi})$ & $\rho(\dot{T}^{(2)}\dot{\Pi})$ & $\rho(\dot{T}^{(3)}\dot{\Pi})$ & $\rho(\dot{T}^{(4)}\dot{\Pi})$   \\
\hline
0.1 & 0.1 & \qquad & -0.0000 & -0.0083 & -0.0149 & -0.0201 & -0.0242  \\
0.1 & 0.5 & \qquad & -0.0000 & -0.0525 & -0.0708 & -0.0783 & -0.0816  \\
0.1 & 0.9 & \qquad & -0.0000 & -0.1215 & -0.1268 & -0.1273 & -0.1273  \\
0.5 & 0.1 & \qquad & -0.0000 & -0.0525 & -0.0708 & -0.0783 & -0.0816  \\
0.5 & 0.5 & \qquad & -0.0000 & -0.2952 & -0.3300 & -0.3375 & -0.3393  \\
0.5 & 0.9 & \qquad & -0.0000 & -0.5419 & -0.5497 & -0.5500 & -0.5500  \\
0.9 & 0.1 & \qquad & -0.0000 & -0.1215 & -0.1268 & -0.1273 & -0.1273  \\
0.9 & 0.5 & \qquad & -0.0000 & -0.5419 & -0.5497 & -0.5500 & -0.5500  \\
0.9 & 0.9 & \qquad & -0.0000 & -0.8629 & -0.8646 & -0.8646 & -0.8646  \\
\hline
\end{tabular}

\vspace{3mm}

\begin{tabular}{|rrrrrrrr|}
\hline\hline
$\alpha$ & $\beta$ & \qquad & $\rho(\dot{M})$ & $\rho(\dot{T}^{(1)} \dot{M})$ & $\rho(\dot{T}^{(2)} \dot{M})$ & $\rho(\dot{T}^{(3)} \dot{M})$ & $\rho(\dot{T}^{(4)} \dot{M})$ \\
\hline
0.1 & 0.1 & \qquad & -1.0000 & -0.8650 & -0.7387 & -0.6233 & -0.5200  \\
0.1 & 0.5 & \qquad & -1.0000 & -0.5610 & -0.2996 & -0.1747 & -0.1218  \\
0.1 & 0.9 & \qquad & -1.0000 & -0.2154 & -0.1345 & -0.1279 & -0.1274  \\
0.5 & 0.1 & \qquad & -1.0000 & -0.5610 & -0.2996 & -0.1747 & -0.1218  \\
0.5 & 0.5 & \qquad & -1.0000 & -0.4667 & -0.3633 & -0.3450 & -0.3411  \\
0.5 & 0.9 & \qquad & -1.0000 & -0.5611 & -0.5505 & -0.5501 & -0.5501  \\
0.9 & 0.1 & \qquad & -1.0000 & -0.2154 & -0.1345 & -0.1279 & -0.1274  \\
0.9 & 0.5 & \qquad & -1.0000 & -0.5611 & -0.5505 & -0.5501 & -0.5501  \\
0.9 & 0.9 & \qquad & -1.0000 & -0.8650 & -0.8646 & -0.8646 & -0.8646  \\
\hline
\end{tabular}

\vspace{3mm}

\begin{tabular}{|rrrrrrrr|}
\hline\hline
$\alpha$ & $\beta$ & \qquad & $\rho(\dot{W})$ & $\rho(\dot{T}^{(1)} \dot{W})$ & $\rho(\dot{T}^{(2)} \dot{W})$ & $\rho(\dot{T}^{(3)} \dot{W})$ & $\rho(\dot{T}^{(4)} \dot{W})$  \\
\hline
0.1 & 0.1 & \qquad & 1.0000 & 0.8548 & 0.7350 & 0.6349 & 0.5502  \\
0.1 & 0.5 & \qquad & 1.0000 & 0.4891 & 0.2249 & 0.0794 & 0.0008  \\
0.1 & 0.9 & \qquad & 1.0000 & 0.0037 & -0.1135 & -0.1259 & -0.1272  \\
0.5 & 0.1 & \qquad & 1.0000 & 0.4891 & 0.2249 & 0.0794 & 0.0008  \\
0.5 & 0.5 & \qquad & 1.0000 & 0.0167 & -0.1859 & -0.2677 & -0.3049  \\
0.5 & 0.9 & \qquad & 1.0000 & -0.4639 & -0.5406 & -0.5491 & -0.5500  \\
0.9 & 0.1 & \qquad & 1.0000 & 0.0037 & -0.1135 & -0.1259 & -0.1272  \\
0.9 & 0.5 & \qquad & 1.0000 & -0.4637 & -0.5406 & -0.5491 & -0.5500  \\
0.9 & 0.9 & \qquad & 1.0000 & -0.8302 & -0.8613 & -0.8643 & -0.8645  \\
\hline
\end{tabular}

\vspace{3mm}

\begin{tabular}{|rrrrrrrr|}
\hline\hline
$\alpha$ & $\beta$ & \qquad & $\rho(\dot{K})$ & $\rho(\dot{T}^{(1)}\dot{K})$ & $\rho(\dot{T}^{(2)}\dot{K})$ & $\rho(\dot{T}^{(3)}\dot{K})$ & $\rho(\dot{T}^{(4)}\dot{K})$   \\
\hline
0.1 & 0.1 & \qquad & -0.6844 & -0.5775 & -0.4828 & -0.4007 & -0.3310 \\
0.1 & 0.5 & \qquad & -0.6844 & -0.3652 & -0.2052 & -0.1354 & -0.1061 \\
0.1 & 0.9 & \qquad & -0.6844 & -0.1754 & -0.1314 & -0.1277 & -0.1273 \\
0.5 & 0.1 & \qquad & -0.6844 & -0.3652 & -0.2052 & -0.1354 & -0.1061 \\
0.5 & 0.5 & \qquad & -0.6844 & -0.3988 & -0.3519 & -0.3426 & -0.3406 \\
0.5 & 0.9 & \qquad & -0.6844 & -0.5544 & -0.5502 & -0.5501 & -0.5501 \\
0.9 & 0.1 & \qquad & -0.6844 & -0.1754 & -0.1314 & -0.1277 & -0.1273 \\
0.9 & 0.5 & \qquad & -0.6844 & -0.5544 & -0.5502 & -0.5501 & -0.5501 \\
0.9 & 0.9 & \qquad & -0.6844 & -0.8643 & -0.8646 & -0.8646 & -0.8646 \\
\hline
\end{tabular}

\end{table}
}

We want to observe how two important concordance measures (sometimes seen as coefficients) of copulas, Spearman's rho and Kendall's tau, change when transformation $\dot{T}$ is applied to them. We refer the reader to \cite[Chapter 5]{Nels} and \cite[Section 2.4]{DuSe} for definitions and explanation of these coefficients. Although formula \eqref{inverse_maxmin_final} is simpler than \cite[(2.3)]{DuOmOrRu} as we pointed out, computing the two coefficients is still a nontrivial task.  In the paper \cite{OmRu}, they were computed only for the case of independent shocks, of course. To go for dependent shocks we need to use the formulas given in \cite{FrNe}:
\[
    \rho(C)=12\int_{0}^{1}\int_{0}^{1}C(u,v)\,du\,dv-3
\]
and
\[
    \tau(C)=4\int_{0}^{1}\int_{0}^{1}C(u,v)\,dC(u,v)-1=1-4 \int_{0}^{1}\int_{0}^{1}C'_u(u,v)C'_v(u,v)\,du\,dv.
\]
Here $C'_t$ is the derivative of $C$ with respect to variable $t$. We applied these formulas to generators $f(u)=u^{1-a}-u$ and $g(v)=v^{1-b}-v$, for all combinations of parameters $a,b \in \{0.1, 0.5, 0.9\}$, and for copulas $\Pi, W$, and $M$. We computed the integrals using the exact expression for copulas and by numerical integration using the Mathematica software \cite{math}. The results are presented in Tables 1 and 2.
The presented experiment suggests that the convergence of the iteration studied in this section is substantially faster with parameters $\alpha$ and $\beta$ big.

{

\begin{table}
  \caption{Kendall's tau for copulas $\Pi,M$, $W$ and Clayton copula $K_{-0.7}$}
  \centering
\begin{tabular}{|rrrrrrrr|}
\hline\hline
$\alpha$ & $\beta$ & \qquad & $\tau(\dot{\Pi})$ & $\tau(\dot{T}^{(1)}\dot{\Pi})$ & $\tau(\dot{T}^{(2)}\dot{\Pi})$ & $\tau(\dot{T}^{(3)}\dot{\Pi})$ & $\tau(\dot{T}^{(4)}\dot{\Pi})$ \\
\hline
0.1 & 0.1 & \qquad & 0.0000 & -0.0055 & -0.0099 & -0.0134 & -0.0162 \\
0.1 & 0.5 & \qquad & 0.0000 & -0.0351 & -0.0473 & -0.0523 & -0.0545 \\
0.1 & 0.9 & \qquad & 0.0000 & -0.0837 & -0.0871 & -0.0874 & -0.0874 \\
0.5 & 0.1 & \qquad & 0.0000 & -0.0351 & -0.0473 & -0.0523 & -0.0545 \\
0.5 & 0.5 & \qquad & 0.0000 & -0.2111 & -0.2338 & -0.2387 & -0.2399 \\
0.5 & 0.9 & \qquad & 0.0000 & -0.4368 & -0.4410 & -0.4412 & -0.4412 \\
0.9 & 0.1 & \qquad & 0.0000 & -0.0837 & -0.0871 & -0.0874 & -0.0874 \\
0.9 & 0.5 & \qquad & 0.0000 & -0.4372 & -0.4410 & -0.4412 & -0.4412 \\
0.9 & 0.9 & \qquad & 0.0000 & -0.8063 & -0.8064 & -0.8064 & -0.8064 \\
\hline
\end{tabular}

\vspace{3mm}

\begin{tabular}{|rrrrrrrr|}
\hline\hline
$\alpha$ & $\beta$ & \qquad & $\tau(\dot{M})$ & $\tau(\dot{T}^{(1)}\dot{M})$ & $\tau(\dot{T}^{(2)} \dot{M})$ & $\tau(\dot{T}^{(3)} \dot{M})$ & $\tau(\dot{T}^{(4)} \dot{M})$ \\
\hline
0.1 & 0.1 & \qquad & -1.0000 & -0.8065 & -0.6457 & -0.5139 & -0.4073 \\
0.1 & 0.5 & \qquad & -1.0000 & -0.4475 & -0.2118 & -0.1184 & -0.0817 \\
0.1 & 0.9 & \qquad & -1.0000 & -0.1495 & -0.0923 & -0.0879 & -0.0875 \\
0.5 & 0.1 & \qquad & -1.0000 & -0.4475 & -0.2118 & -0.1184 & -0.0817 \\
0.5 & 0.5 & \qquad & -1.0000 & -0.3333 & -0.2561 & -0.2437 & -0.2411 \\
0.5 & 0.9 & \qquad & -1.0000 & -0.4474 & -0.4415 & -0.4412 & -0.4412 \\
0.9 & 0.1 & \qquad & -1.0000 & -0.1495 & -0.0923 & -0.0879 & -0.0875 \\
0.9 & 0.5 & \qquad & -1.0000 & -0.4474 & -0.4415 & -0.4412 & -0.4412 \\
0.9 & 0.9 & \qquad & -1.0000 & -0.8065 & -0.8064 & -0.8064 & -0.8064 \\
\hline
\end{tabular}

\vspace{3mm}

  \begin{tabular}{|rrrrrrrr|}
\hline\hline
$\alpha$ & $\beta$ & \qquad & $\tau(\dot{W})$ & $\tau(\dot{T}^{(1)} \dot{W})$ & $\tau(\dot{T}^{(2)} \dot{W})$ & $\tau(\dot{T}^{(3)} \dot{W})$ & $\tau(\dot{T}^{(4)} \dot{W})$ \\
\hline
0.1 & 0.1 & \qquad & 1.0000 & 0.8000 & 0.6540 & 0.5428 & 0.4553 \\
0.1 & 0.5 & \qquad & 1.0000 & 0.4000 & 0.1692 & 0.0577 & 0.0015 \\
0.1 & 0.9 & \qquad & 1.0000 & 0.0001 & -0.0810 & -0.0875 & -0.0875 \\
0.5 & 0.1 & \qquad & 1.0000 & 0.4000 & 0.1692 & 0.0577 & 0.0015 \\
0.5 & 0.5 & \qquad & 1.0000 & 0.0000 & -0.1413 & -0.1951 & -0.2187 \\
0.5 & 0.9 & \qquad & 1.0000 & -0.4000 & -0.4368 & -0.4412 & -0.4412 \\
0.9 & 0.1 & \qquad & 1.0000 & 0.0001 & -0.0810 & -0.0875 & -0.0875 \\
0.9 & 0.5 & \qquad & 1.0000 & -0.4000 & -0.4368 & -0.4412 & -0.4412 \\
0.9 & 0.9 & \qquad & 1.0000 & -0.8000 & -0.8058 & -0.8064 & -0.8064 \\
\hline
\end{tabular}

\vspace{3mm}

\begin{tabular}{|rrrrrrrr|}
\hline\hline
$\alpha$ & $\beta$ & \qquad & $\tau(\dot{K})$ & $\tau(\dot{T}^{(1)}\dot{K})$ & $\tau(\dot{T}^{(2)} \dot{K})$ & $\tau(\dot{T}^{(3)} \dot{K})$ & $\tau(\dot{T}^{(4)} \dot{K})$ \\
\hline
0.1 & 0.1 & \qquad & -0.5385 & -0.4368 & -0.3532 & -0.2854 & -0.2308 \\
0.1 & 0.5 & \qquad & -0.5385 & -0.2569 & -0.1390 & -0.0909 & -0.0710 \\
0.1 & 0.9 & \qquad & -0.5385 & -0.1202 & -0.0902 & -0.0877 & -0.0875 \\
0.5 & 0.1 & \qquad & -0.5385 & -0.2569 & -0.1390 & -0.0909 & -0.0710 \\
0.5 & 0.5 & \qquad & -0.5385 & -0.2810 & -0.2483 & -0.2421 & -0.2407 \\
0.5 & 0.9 & \qquad & -0.5385 & -0.4436 & -0.4413 & -0.4412 & -0.4412 \\
0.9 & 0.1 & \qquad & -0.5385 & -0.1202 & -0.0902 & -0.0877 & -0.0875 \\
0.9 & 0.5 & \qquad & -0.5385 & -0.4436 & -0.4413 & -0.4412 & -0.4412 \\
0.9 & 0.9 & \qquad & -0.5385 & -0.8064 & -0.8064 & -0.8064 & -0.8064 \\
\hline
\end{tabular}
\end{table}
}

Next we compute analytically the tail dependence coefficients for the limit reflected maxmin copula,
see \cite[Section 5.4]{Nels} and \cite[Section 2.6.1]{DuSe}. One may find there an equivalent expression for each of the coefficients
\[
\lambda_L=\lim_{t\downarrow0}\frac{\delta_C(t)}{t}\quad\mbox{and}\quad \lambda_U=\lim_{t\uparrow1}\frac{1-2t+\delta_C(t)}{1-t}.
\]

\begin{proposition}
  \begin{description}
    \item[(a)] If $\alpha=0$ and $\beta=0$, then $\lambda_L ({\overline{C}_{f,g}}^\sigma)= \lambda_L(\dot{C})$, if not, then $\lambda_L ({\overline{C}_{f,g}}^\sigma)=0$.
    \item[(b)] If $\alpha$ and $\beta$ are both strictly smaller than 1, then $\lambda_U({\overline{C}_{f,g}}^\sigma)= \lambda_U(\dot{C})$, if not, then $\lambda_U({\overline{C}_{f,g}}^\sigma)=0$.
    \item[(c)] If either $\alpha\neq 0$ and $\beta\neq 1$ or $\alpha=0$ and $\beta=1$ then $\lambda_L(\overline{C}_{\phi,\psi}) = 0$. If $\alpha= 0$ and $\beta \neq 1$ then $\lambda_L(\overline{C}_{\phi,\psi}) = \lambda_L(C)$.  In the remaining case $\alpha\neq 0$ and $\beta = 1$ the value of $\lambda_L(\overline{C}_{\phi,\psi})$ depends on generators $\phi$ and $\psi$.
    \item[(d)] If either $\alpha\neq 1$ and $\beta\neq 0$ or $\alpha=1$ and $\beta=0$ then $\lambda_U(\overline{C}_{\phi,\psi}) = 0$. If $\alpha\neq 1$ and $\beta = 0$ then $\lambda_U(\overline{C}_{\phi,\psi}) = \lambda_U(C)$.  In the remaining case $\alpha=1$ and $\beta\neq 0$ the value of $\lambda_U(\overline{C}_{\phi,\psi})$ depends on generators $\phi$ and $\psi$.
  \end{description}
\end{proposition}

\begin{proof} Straightforward computation.
\end{proof}

At this point it is appropriate to emphasize that in the situation of the semi-group action interpretation of the Durante-Girard-Mazo model presented in \cite[Section 3]{DuGiMa} (cf.\ also \cite{DuGiMa1}) some concrete stability results (cf. \cite[Example 3.1]{DuGiMa}) show similar behavior to what we observed in the case of maxmin and reflected maxmin transformations given in Theorems \ref{thm:main} \& \ref{thm:maxmin}. We think that these similarities are more than just coincidental and deserve to be studied further.

\section{ Multivariate reflected maxmin copulas: $n=3$  }\label{sec:multi3}

It is now time to move to the second main result of the paper, the extension of the dependent reflected maxmin copulas to the multivariate case. In order to make our computations easier to understand, let us first compute one of the two possible cases for three variables. To get a perception of the possible improvement of reflected maxmin approach compared to the original maxmin techniques consider formula on \cite[p.\ 166]{DuOmOrRu}, where two minima and one maximum are taken

\begin{align*}
  T_{\phi_1,\phi_2,\phi_3}(C)(u_1,u_2,u_3) & = C(\phi_1(u_1),\phi_2(u_2), \phi_3(u_3)) \max\{0,\phi_1^*(u_1)-\max\{\phi_2^*(u_2),\phi_3^*(u_3)\}\}
 \\
   & +C(\phi_1(u_1), 1,\phi_3(u_3))\max\{0, \min\{\phi_1^*(u_1), \phi_2^*(u_2)\} - \phi_3^*(u_3)\}
 \\
   & +C(\phi_1(u_1),\phi_2(u_2), 1)\max\{0, \min\{\phi_1^*(u_1),\phi_3^* (u_3)\} - \phi_2^*(u_2)\}
 \\
   & +\phi_1(u_1) \min\{\phi_1^*(u_1),\phi_2^* (u_2),\phi_3^* (u_3)\}.
\end{align*}

We will rewrite this formula into Equation \eqref{sqi_n3p1} in the reflected maxmin setting. An interested reader may find the computations leading to this formula, although somewhat tedious, in the next lines. Our point is that it may be helpful to understand better the main steps of the computation in smaller dimension before we extend these steps to the general level in the next section. First, rewrite the above formula as
\begin{equation}\label{n3p1}
  T_{\phi_1,\phi_2,\phi_3}(C)(u_1,u_2,u_3)= A_1+A_2+A_3+A_4,
\end{equation}
where
\begin{align*}
    \begin{split}
    A_1& = C(\phi_1(u_1), \phi_2(u_2), \phi_3(u_3)) \max\{0,\phi_1^*(u_1)- \max\{\phi_2^*(u_2), \phi_3^*(u_3)\}\} \\
    A_2& = C(\phi_1(u_1), 1, \phi_3(u_3)) \max\{0, \min\{\phi_1^*(u_1), \phi_2^*(u_2)\} - \phi_3^*(u_3)\} \\
    A_3& =C(\phi_1(u_1),\phi_2(u_2), 1) \max\{0, \min\{\phi_1^*(u_1), \phi_3^*(u_3)\} - \phi_2^*(u_2)\} \\
   A_4& = C(\phi_1(u_1),1,1) \min\{\phi_1^*(u_1), \phi_2^*(u_2), \phi_3^*(u_3)\},
    \end{split}
\end{align*}
We want to perform the two flips on the copula $T(\cdot)$ respectively on terms $A_i$. First transform the term $A_1$ by a flip in the second variable, i.e.\ $u_2\mapsto 1-u_2$:
\begin{align*}
    A_1& \mapsto C(\phi_1(u_1), \phi_2(1), \phi_3(u_3))
\max\{0,\min\{\phi_1^*(u_1)- \phi_2^* (1), \phi_1^*(u_1)-\phi_3^* (u_3)\}\}\\
  & -C(\phi_1(u_1), \phi_2(1-u_2), \phi_3(u_3))\max\{0,\min\{\phi_1^*(u_1)- \phi_2^* (1-u_2), \phi_1^*(u_1)-\phi_3^* (u_3)\}\}\\
    &=-C(\phi_1(u_1), \phi_2(1-u_2), \phi_3(u_3)) \max\{0, \min\{\phi_1^*(u_1)- \phi_2^* (1-u_2), \phi_1^*(u_1)-\phi_3^* (u_3)\}\}.
\end{align*}
Next, we perform a flip in the third variable, i.e.\ the transformation $u_3\mapsto 1-u_3$, on $A_1$:
\begin{align*}
    A_1& \mapsto -C(\phi_1(u_1), \phi_2(1-u_2), \phi_3(1))\max\{0, \min\{\phi_1^*(u_1)- \phi_2^* (1-u_2), \phi_1^*(u_1)-\phi_3^* (1)\}\}\\
  & +C(\phi_1(u_1), \phi_2(1-u_2), \phi_3(1-u_3))\max\{0, \min\{\phi_1^*(u_1)- \phi_2^* (1-u_2), \phi_1^*(u_1)-\phi_3^* (1-u_3)\}\}\\
    &=C(\phi_1(u_1), \phi_2(1-u_2), \phi_3(1-u_3))\max\{0, \min\{\phi_1^*(u_1)- \phi_2^* (1-u_2), \phi_1^*(u_1)-\phi_3^* (1-u_3)\}\}.
\end{align*}
We now perform the two flips on $A_2$, first a flip in the second variable, i.e.\ $u_2\mapsto 1-u_2$:
\begin{align*}
  A_2 & \mapsto C(\phi_1(u_1), 1, \phi_3(u_3)) \max\{0, \min\{\phi_1^*(u_1)- \phi_3^*(u_3), \phi_2^*(1)- \phi_3^*(u_3))\} \} \\
   & -C(\phi_1(u_1), 1, \phi_3(u_3)) \max\{0, \min\{\phi_1^*(u_1)- \phi_3^*(u_3), \phi_2^*(1-u_2)- \phi_3^*(u_3)\}\} \\
   & =C(\phi_1(u_1), 1, \phi_3(u_3)) \max\{0, \min\{\phi_1^*(u_1)- \phi_3^*(u_3), \phi_1^*(u_1)-\phi_2^*(1-u_2)\} \}.
\end{align*}
Here, we have used the fact that the second term in the first $\min$ above is clearly greater than the first one and can therefore be neglected, and then we have used an obvious rule
\begin{equation}\label{rule}
  \max\{0,\alpha\}-\max\{0,\min\{\alpha,\beta\}\}=\max\{0,\min\{\alpha, \alpha-\beta\}\},\ \ \mbox{for}\ \ \alpha,\beta\in\mathds{R},
\end{equation}
to simplify the expression containing the two minima. Now, we apply a flip in the third variable, i.e.\ the transformation $u_3\mapsto 1-u_3$, on $A_2$:
\begin{align*}
  A_2 & \mapsto C(\phi_1(u_1), 1, \phi_3(1)) \max\{0, \min\{\phi_1^*(u_1)- \phi_3^*(1), \phi_1^*(u_1)-\phi_2^*(1-u_2)\} \} \\
   & -C(\phi_1(u_1), 1, \phi_3(1-u_3)) \max\{0, \min\{\phi_1^*(u_1)- \phi_3^*(1-u_3), \phi_1^*(u_1)-\phi_2^*(1-u_2)\} \} \\
   & =-C(\phi_1(u_1), 1, \phi_3(1-u_3)) \max\{0, \min\{\phi_1^*(u_1)- \phi_3^*(1-u_3), \phi_1^*(u_1)-\phi_2^*(1-u_2)\} \}.
\end{align*}
It is clear that $A_3$ can be obtained from $A_2$ by exchanging indices $2$ and $3$, so after performing the two transformations $u_2\mapsto 1-u_2$ and  $u_3\mapsto 1-u_3$ (in any order), we get by analogy with the result above:
\begin{align*}
  A_3 & \mapsto \\
   & -C(\phi_1(u_1), \phi_2(1-u_2),1) \max\{0, \min\{\phi_1^*(u_1)- \phi_3^*(1-u_3), \phi_1^*(u_1)-\phi_2^*(1-u_2)\} \}.
\end{align*}
It remains to transform the last term $A_4$, first by flipping the second variable, i.e.\ $u_2\mapsto 1-u_2$:
\begin{align*}
  A_4 & \mapsto C(\phi_1(u_1),1,1) \min\{\phi_1^*(u_1), \phi_2^*(1), \phi_3^*(u_3)\} \\
   & -C(\phi_1(u_1),1,1) \min\{\phi_1^*(u_1), \phi_2^*(1-u_2), \phi_3^*(u_3)\} \\
   & =-C(\phi_1(u_1),1,1) \max\{0,\min\{\phi_1^*(u_1), \phi_3^*(u_3), \phi_1^*(u_1)-\phi_2^*(1-u_2),\\
   & \hskip 20em \phi_3^*(u_3)-\phi_2^*(1-u_2)\}\}\\
   & =-C(\phi_1(u_1),1,1) \max\{0,\min\{ \phi_1^*(u_1)-\phi_2^*(1-u_2), \\  & \hskip 20em \phi_3^*(u_3)-\phi_2^*(1-u_2)\}\}.
\end{align*}
In the considerations above, we have first neglected $\phi_2^*(1)=1$, clearly no smaller than the other two compared quantities of the first $\min$; then we have used Rule \eqref{rule} with $\alpha=\min\{ \phi_1^*(u_1), \phi_3^*(u_3)\}$ and $\beta=\phi_3^*(1-u_3)$.
And finally we flip the third variable, i.e.\ $u_3\mapsto 1-u_3$ in $A_4$:
\begin{align*}
  A_4 & \mapsto -C(\phi_1(u_1),1,1) \max\{0,\min\{ \phi_1^*(u_1)-\phi_2^*(1-u_2), \phi_3^*(1)-\phi_2^*(1-u_2)\}\} \\
   & +C(\phi_1(u_1),1,1) \max\{0,\min\{ \phi_1^*(u_1)-\phi_2^*(1-u_2), \phi_3^*(1-u_3)-\phi_2^*(1-u_2)\}\} \\
   & =C(\phi_1(u_1),1,1) \max\{0,\min\{ \phi_1^*(u_1)-\phi_2^*(1-u_2), \phi_1^*(u_1)-\phi_3^*(1-u_3)\}\}
\end{align*}
and use the above considerations. We have thus transformed maxmin copula $T$ of \eqref{n3p1} into reflected maxmin copula
\begin{equation}\label{inverse_n3p1}
\begin{split}
     \dot{T_{\phi_1,\phi_2,\phi_3}}(C)(u_1,u_2,u_3) &\mapsto \max\{0, \min\{ \phi_1^*(u_1)-\phi_2^*(1-u_2), \phi_1^*(u_1)-\phi_3^*(1-u_3)\}\}\times \\
     & (C(\phi_1(u_1), \phi_2(1-u_2), \phi_3(1-u_3))-C(\phi_1(u_1), 1, \phi_3(1-u_3)) \\
     & -C(\phi_1(u_1),\phi_2(1-u_2),1)+C(\phi_1(u_1),1,1)).
\end{split}
\end{equation}
The second factor above reminds us of the reflected copula related to $C$, but not quite yet. In order do get a better expression we need to replace the generating functions of the maxmin copula and their auxiliary functions with the according reflected maxmin generators and their auxiliaries. It is clear that when we follow \eqref{inverse_generators} we have to use the $\phi\mapsto f$ transformation rule for the generator of index $i=1$ and $\psi\mapsto g$ transformation rule for the generators of indices $j=2,3$:
\begin{align}\label{multi_inverse_generators}
\begin{split}
   f_i(u_i) = \phi_i(u_i) - u_i,\ & f_j(u_j) = 1 - u_j - \phi_j(1 - u_j), \\
   f_i^*(u_i) = \frac{f_i(u_i)}{u_i}, \ \widehat{f}_i(u_i)=u_i+f_i(u_i),\   & f_j^*(u_j) = \frac{f_j(u_j)}{u_j}, \ \widehat{f_j}(u_j)=u_j+f_j(u_j).
\end{split}
\end{align}
Using these definitions, we transform the second factor of \eqref{inverse_n3p1} into
\[
    \dot{C}(\widehat{f}_1(u_1),\widehat{f}_2(u_2),\widehat{f}_3(u_3)),
\]
and the first factor of \eqref{inverse_n3p1} termwise. The first term equals
\[
    \phi_1^*(u_1)-\phi_2^*(1-u_2)=\frac{u_1}{\widehat{f}_1(u_1)}- \frac{f_2(u_2)}{\widehat{f}_2(u_2)}=\frac{u_1u_2-f_1(u_1)f_2(u_2)} {\widehat{f}_1(u_1)\widehat{f}_2(u_2)}
\]
and similarly for the second term, where we replace index $i=2$ with $i=3$. So, with the reflected maxmin generators we can write \eqref{inverse_n3p1} in the form \eqref{sqi_n3p1} as required:
\begin{equation}\label{sqi_n3p1}
    \begin{split}
     &\hspace{35pt}\dot{T_{\widehat{f}_1,\widehat{f}_2,\widehat{f}_3}}(\dot{C})(u_1,u_2,u_3) =\frac{\dot{C}(\widehat{f}_1(u_1),\widehat{f}_2(u_2),\widehat{f}_3(u_3))} {\widehat{f}_1(u_1)\widehat{f}_2(u_2)\widehat{f}_3(u_3)}\times\\
& \max\left\{0,
\min\{ (u_1u_2-f_1(u_1)f_2(u_2)) \widehat{f}_3(u_3), (u_1u_3-f_1(u_1)f_3(u_3)) \widehat{f}_2(u_2)\}\right\}.
    \end{split}
\end{equation}

In Figures 5 to 8 we give examples of 3D scatterplots of reflected maxmin copulas. We explain how the samples of these plots have been generated in the Appendix. Generating functions for the scatterplot in Figure 5 are functions $f_i(u_i)=\frac12(1-u_i)$ and $C$ is the product copula. In the figure there are three different views of the same scatterplot. The copula has 1-dimensional and 2-dimensional singular components. 1-dimensional singular component lies on the curve $u_2=u_3=\dfrac{1-u_1}{1+3u_1}$. 2-dimensional singular components lie on the surfaces $u_2=\dfrac{1-u_1}{1+3u_1}$, $u_3=\dfrac{1-u_1}{1+3u_1}$ and $u_2=u_3$. The zero set of the copula is the set $Z=\{(u_1,u_2,u_3)\in [0,1]^3: \, u_2 \leqslant \dfrac{1-u_1}{1+3u_1} \quad\mathrm{or}\quad u_3 \leqslant \dfrac{1-u_1}{1+3u_1} \}.$ On the complement of the union of the zero set and the singular components  the third mixed derivative of the copula exists and is positive.
\begin{figure}[h!]
  \centering
\begin{center}
\includegraphics[width=3.5cm]{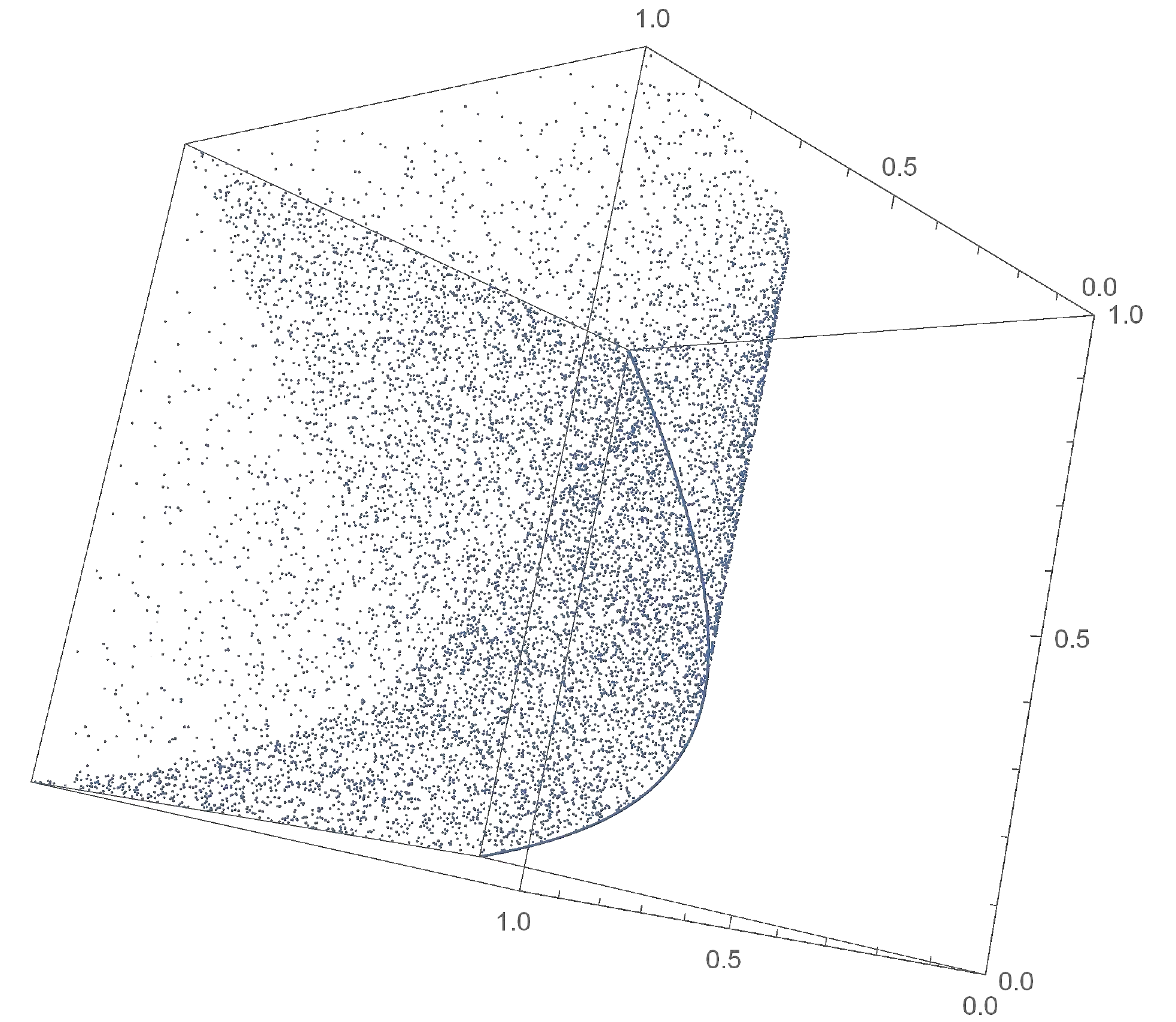} \hfil
\includegraphics[width=3.5cm]{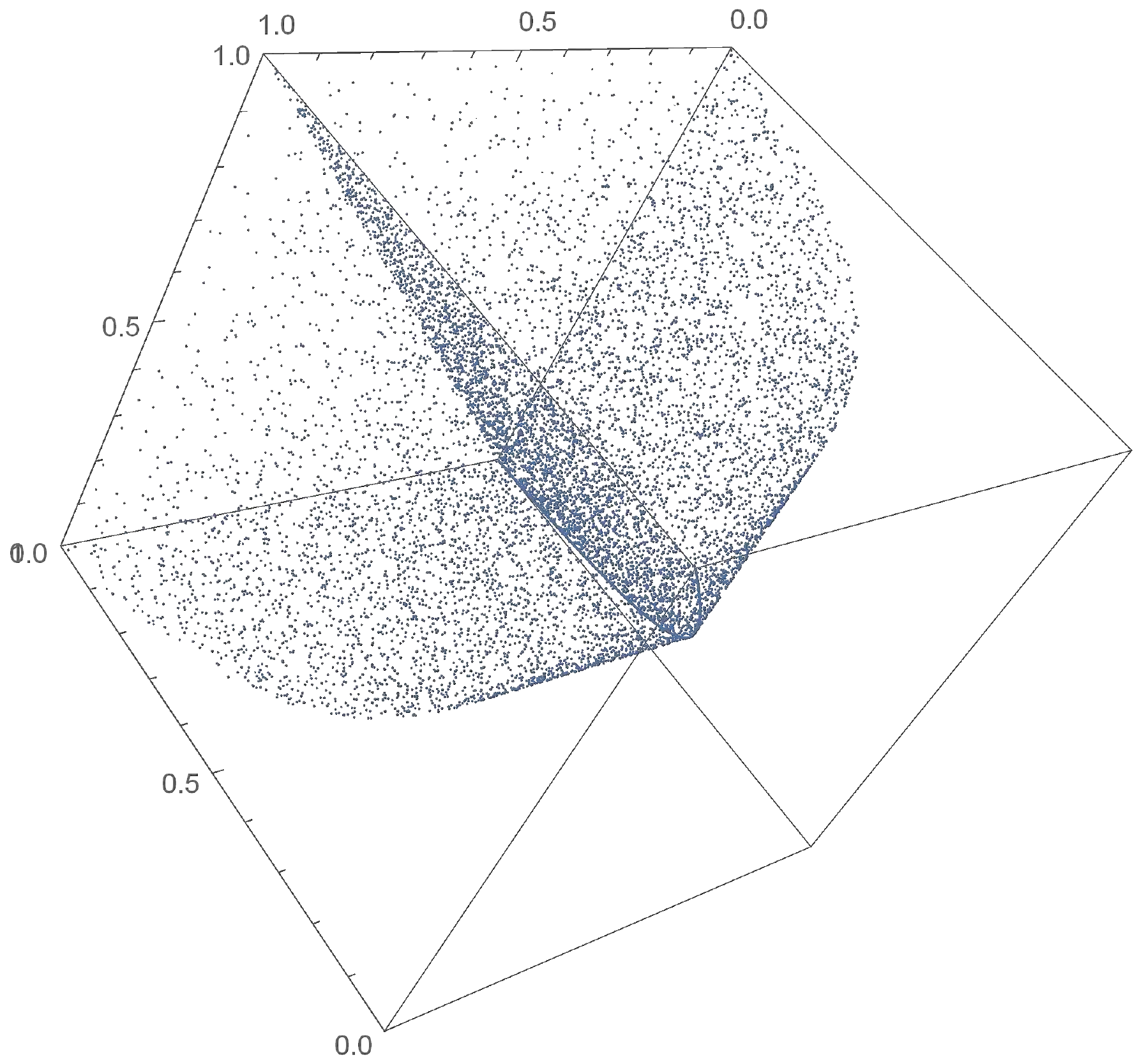} \hfil
\includegraphics[width=3.5cm]{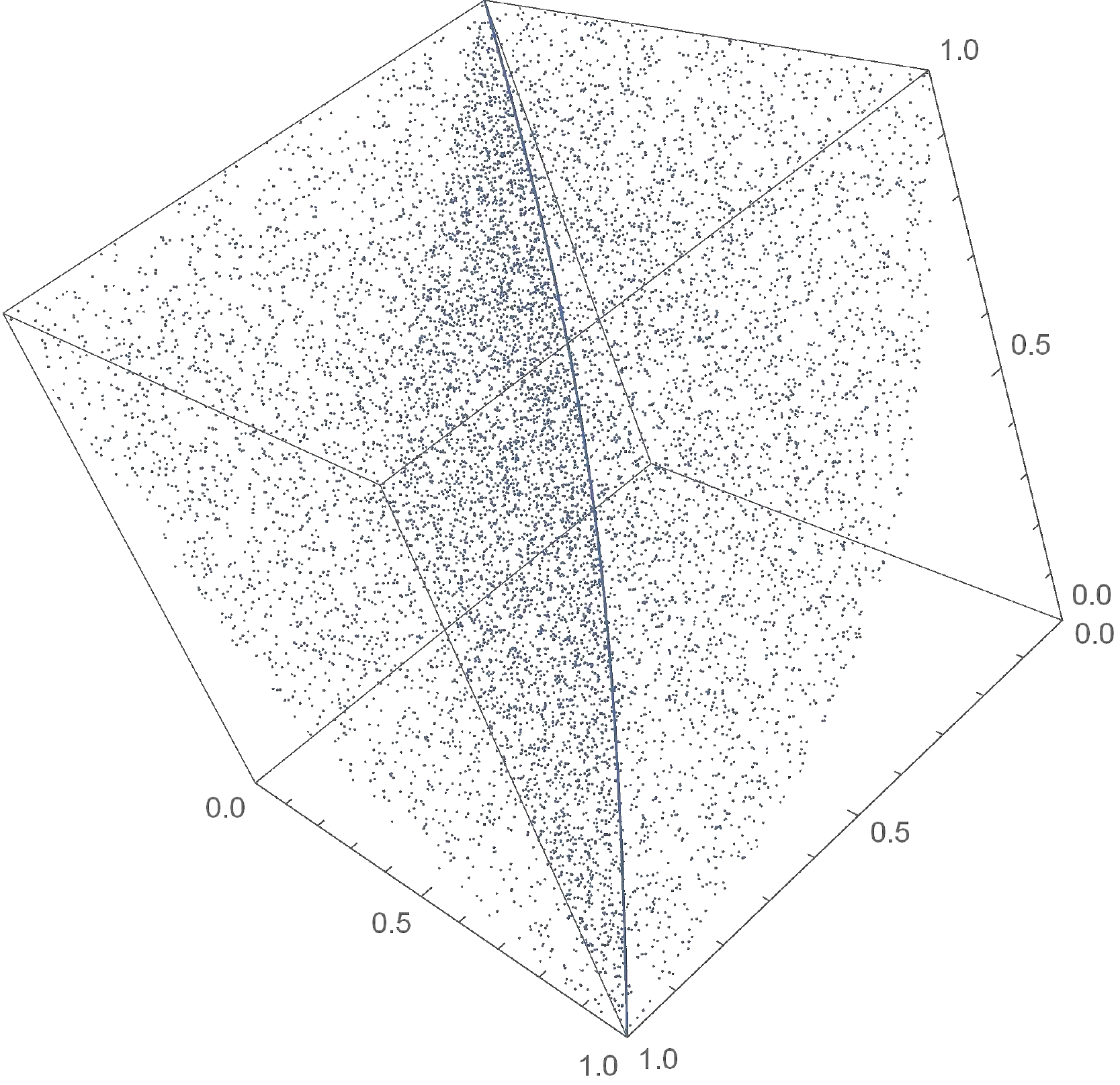}
\caption{3D-Scatterplots of reflected maxmin copulas }
\end{center}
\end{figure}

Generating functions for scatterplot in Figure 6 are functions $f_1(u_1)=\min \{u_1,1-u_1\}$ and $f_i(u_i)=1-u_i$ for $i=2,3$ and $C$ is the product copula. The copula has 1-dimensional and 2-dimensional singular components. 1-dimensional singular component lies on the curve $u_2=u_3=1-u_1$ for $u_1 \geqslant \frac12$. 2-dimensional singular component lies on the surface $u_2=u_3$ for $u_1 \leqslant \frac12$ and $u_2 \geqslant \frac12$. The third mixed derivative of the copula is zero where it exists.

\begin{figure}[h!]
  \centering
  \begin{center}
  \includegraphics[width=5cm]{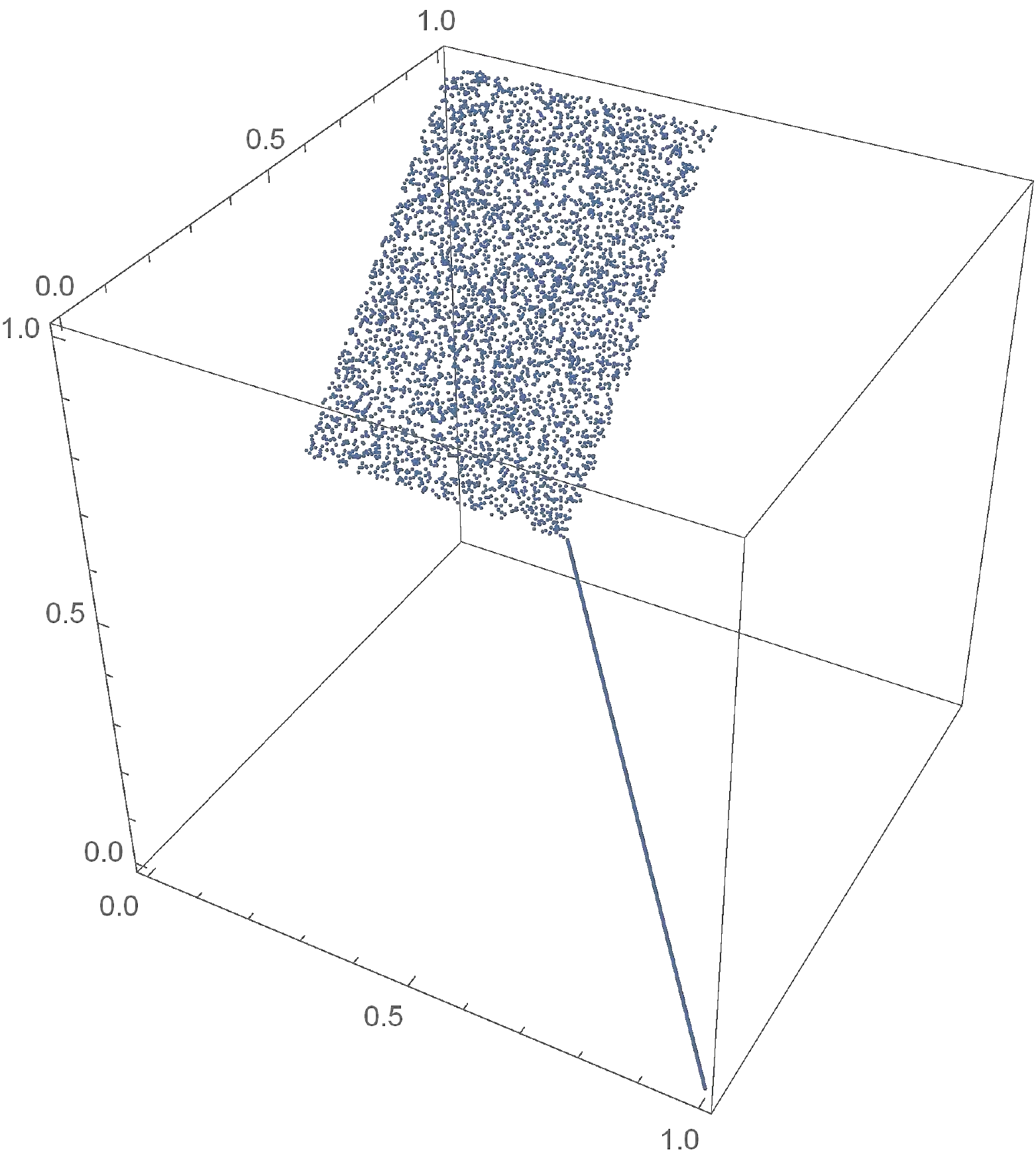}
  \end{center}
\caption{Reflected maxmin copula with a 1-dimensional and 2-dimensional singular component }
\end{figure}

Generating functions for scatterplot in Figure 7 are functions $f_1(u_1)=\min \{u_1,1-u_1\}$, $f_2(u_2)=1-u_2$, $f_3(u_3)=\frac12(1-u_3)$ and $C$ is the product copula. The copula has 1-dimensional and 2-dimensional singular components. 1-dimensional singular component lies on the curve $u_2=1-u_1$, $u_3=\dfrac{1-u_1}{1+u_1}$ for $u_1 \geqslant \frac12$. 2-dimensional singular components lie on the surfaces $u_2=1-u_1$, $u_3 \geqslant \dfrac{1-u_1}{1+u_1}$ for $u_1 \geqslant \frac12$ and $u_3=\frac{u_2}{2-u_2} $ for $u_1\leqslant \frac12$ and $u_2 \geqslant \frac12$. The third mixed derivative of the copula  exists and is positive on the region $\{(u_1,u_2,u_3) \in [0,1]^3: \, u_1<\frac12 \quad \mathrm{and} \quad u_3 >  \dfrac{u_2}{2-u_2}\}$. In the figure there are three different views of the same scatterplot.

\begin{figure}[h!]
    \centering
\begin{center}
\includegraphics[width=3.5cm]{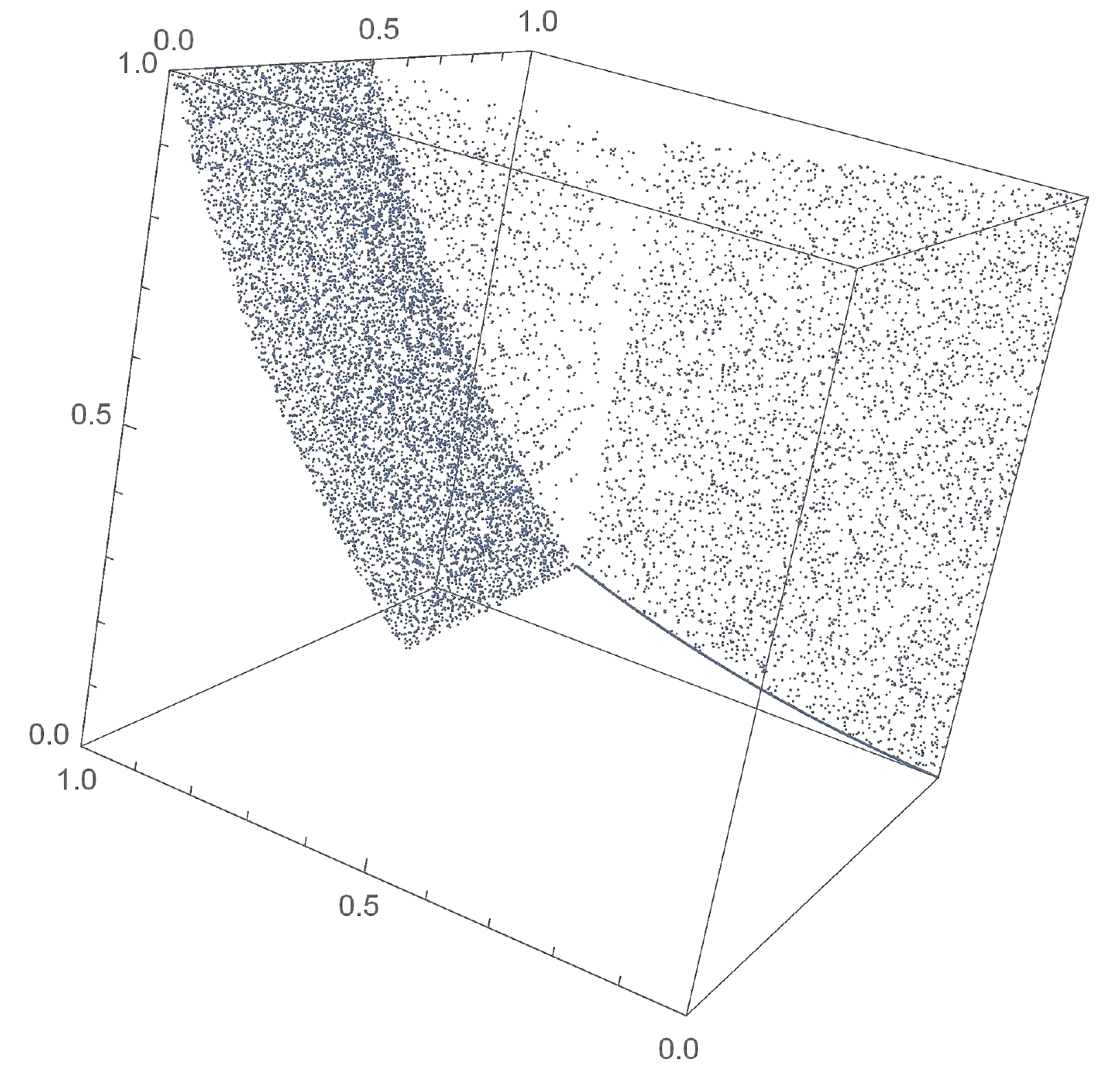}
\includegraphics[width=3.5cm]{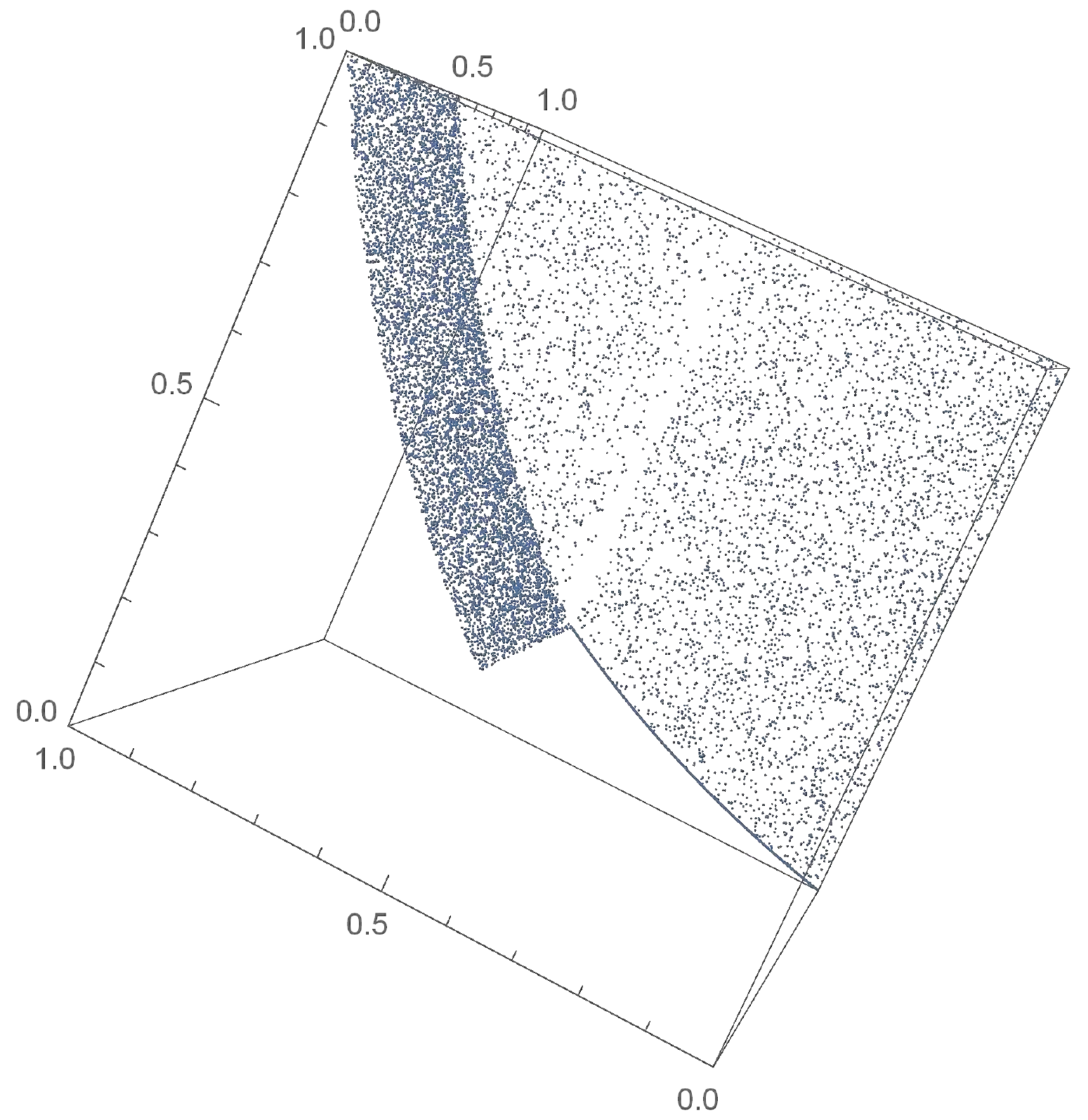}
\includegraphics[width=3.5cm]{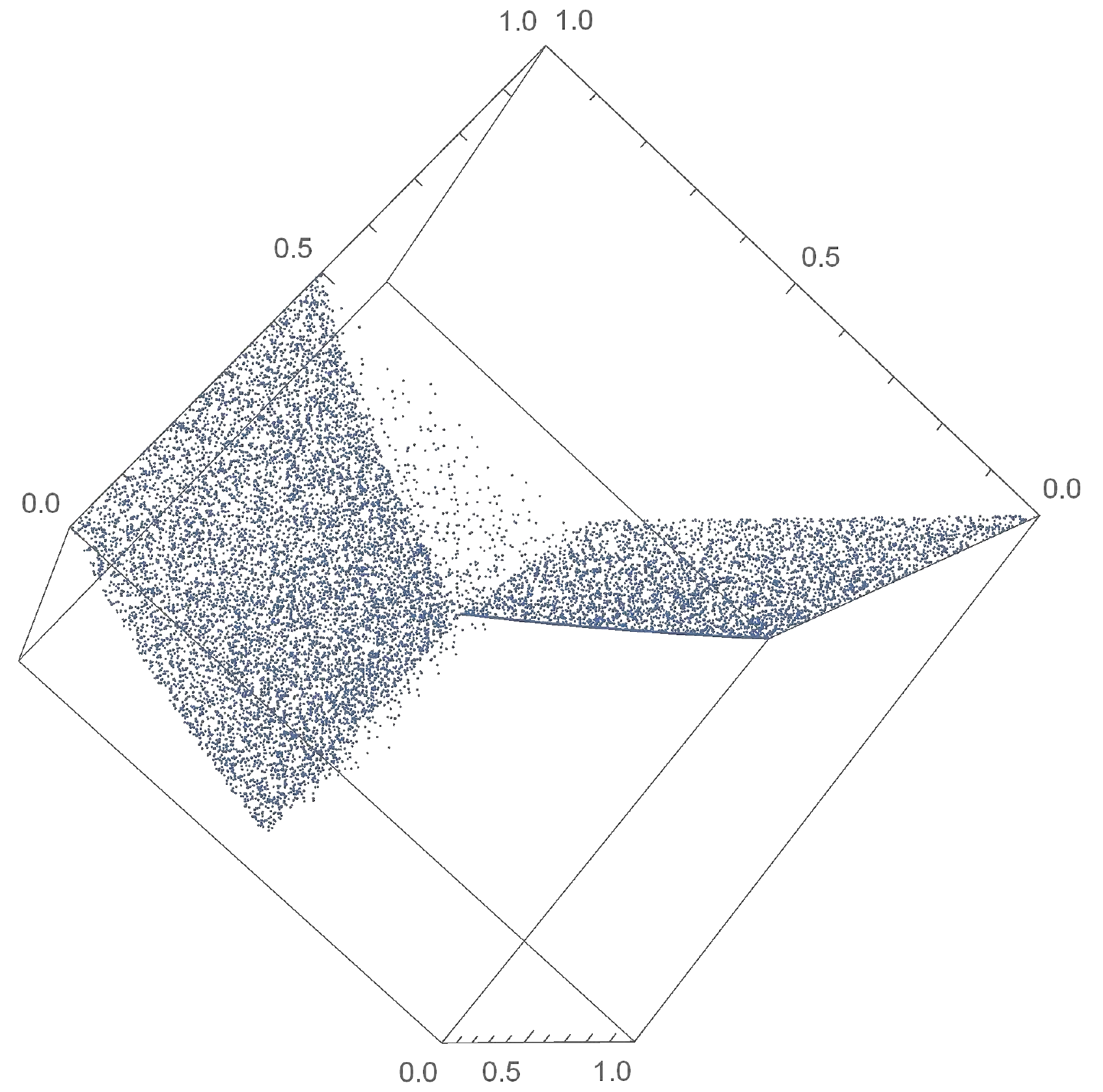}
\end{center}
    \caption{Reflected maxmin copula with a 1-dimensional and two 2-dimensional singular components }
\end{figure}

Generating functions for scatterplot in Figure 8 are functions $f_1(u_1)=\min \{u_1,1-u_1\}$, $f_2(u_2)=1-u_2$, $f_3(u_3)=\frac12(1-u_3)$ and $C(u_1,u_2,u_3)=M(u_1,u_2,u_3)=\min\{u_1,u_2,u_3\}$. The copula has 1-dimensional and 2-dimensional singular components. 1-dimensional singular component lies on the curve $u_2=1-u_1$, $u_3=\dfrac{1-u_1}{1+u_1}$ for $u_1 \geqslant \frac12$. The first 2-dimensional singular component lies on the surface $u_2=1-u_1$, $u_3 \geqslant \dfrac{1-u_1}{1+u_1}$ for $u_1 \geqslant \frac12$, the second lies on $u_2 \geqslant \dfrac{4u_1-1}{2u_1} $, $u_3= \frac{u_2}{2-u_2}$ for $u_1\leqslant \frac12$, $u_2 \geqslant \frac12$, and the third lies on $u_3=4u_1-1$ for $u_1 \leqslant \frac12$, $\frac12\leqslant u_2 \leqslant \dfrac{4u_1-1}{2u_1}$. The third mixed derivative of the copula is zero where it exists. In the figure there are two different views of the same scatterplot.

\begin{figure}[h!]
    \centering
\begin{center}
\begin{center}
\includegraphics[width=4.5cm]{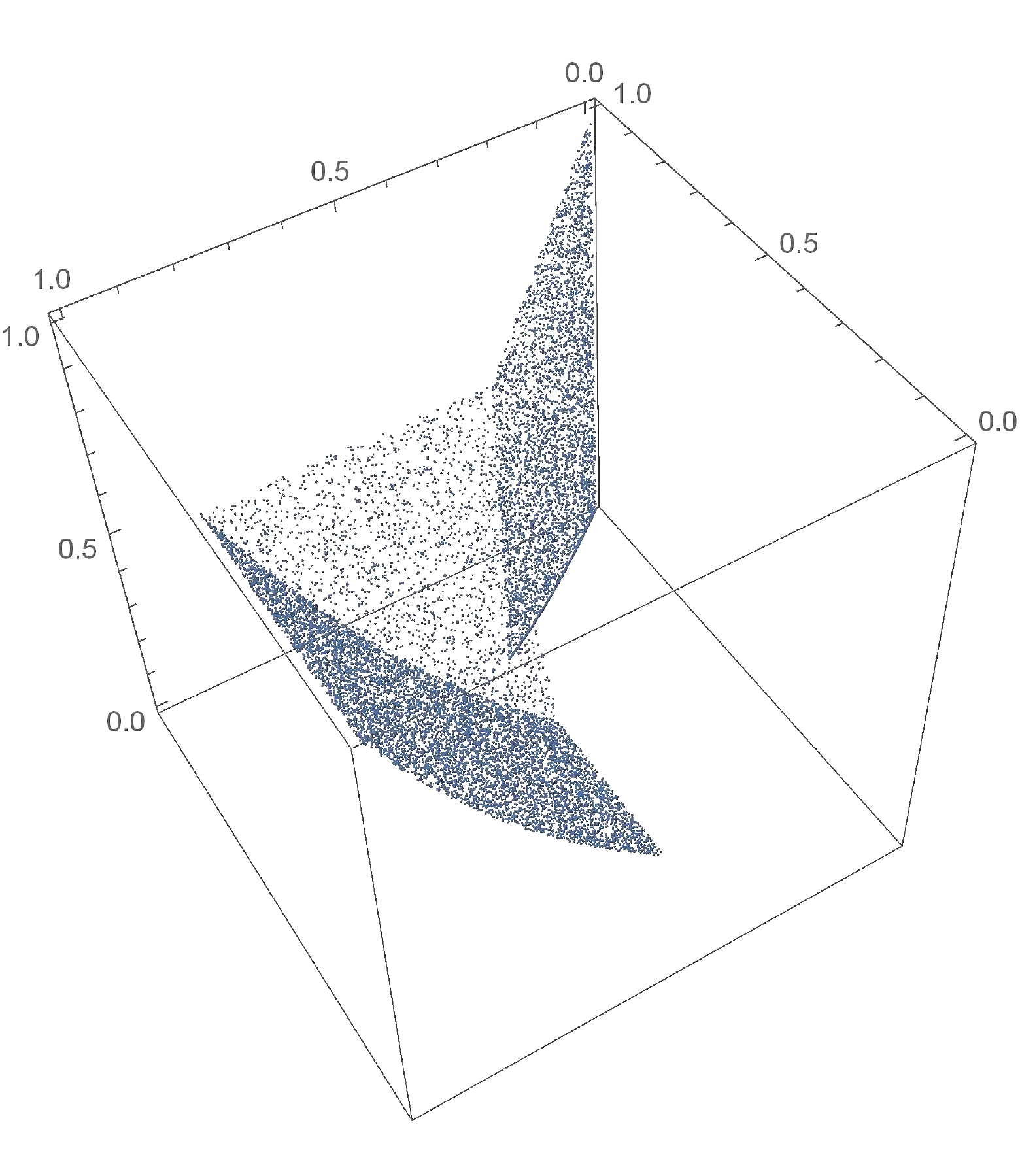}
\includegraphics[width=4.5cm]{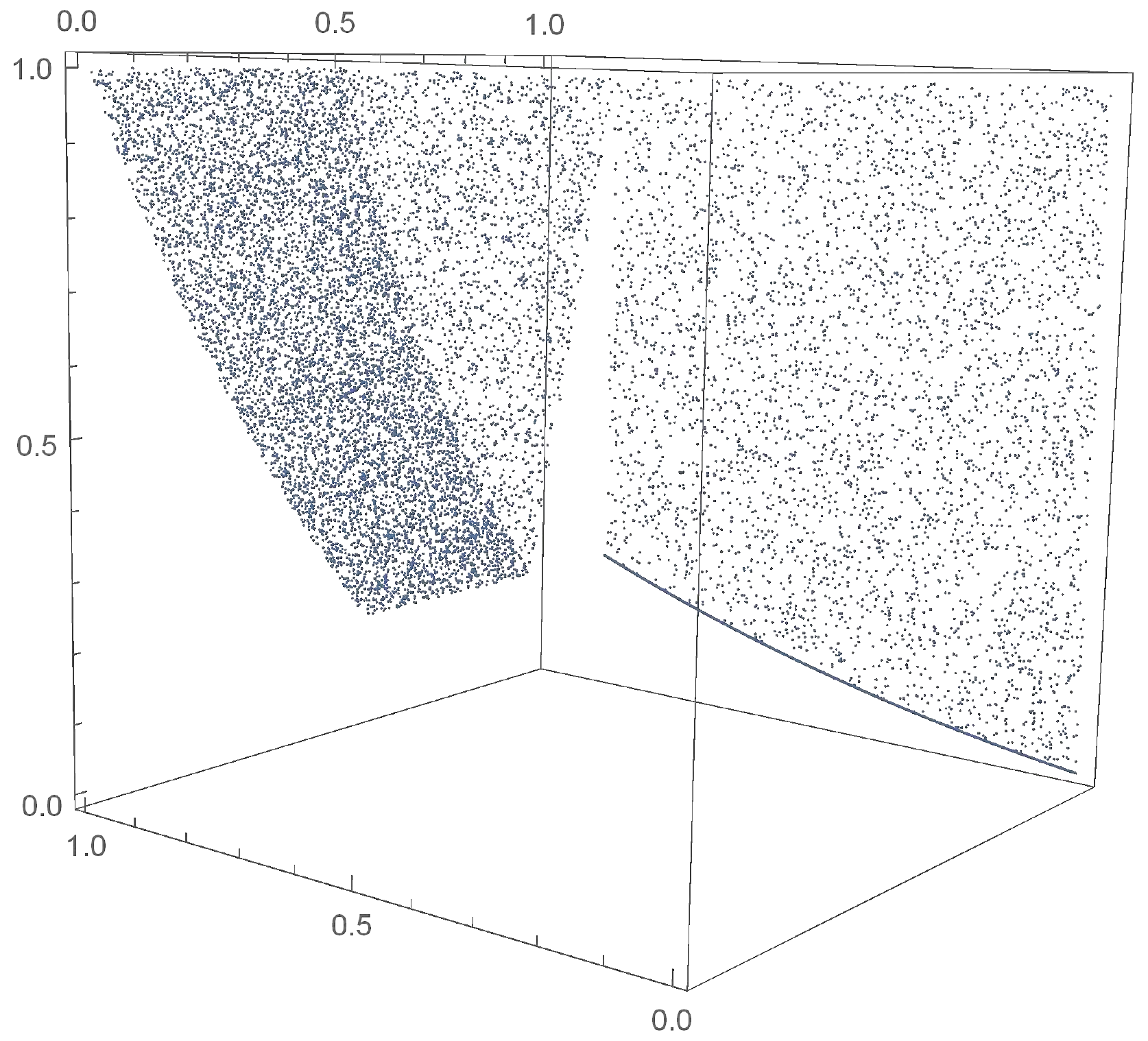}
\end{center}\end{center}
    \caption{Another example of reflected maxmin copula with more singular components }
\end{figure}

After we did all this for one important case of $n=3$, we can go for general $n$.

\section{ Multivariate reflected maxmin copulas: the general case  }\label{sec:multi}

Recal the formula on \cite[p.\ 166]{DuOmOrRu} that we transformed in Section \ref{sec:multi3} into \eqref{sqi_n3p1}. In order to do that for general $n$, we first derive the general formula in the form of \eqref{inverse_n3p1} of Section \ref{sec:multi3}. We need to specify in what sense the inclusion-exclusion principle is made on the copula $C$ (note that we adjust here the notation of \cite{DuOmOrRu} in an obvious way)

\begin{theorem}
Let $T_{\bf \Phi}(C)({\bf u})$ be a maxmin $n$-copula with $p$ maxima and $n-p$ minima. Then its reflected maxmin copula takes the form
$$\dot{T_{\bf \Phi}}(C)({\bf u})=
C^\prime\max \left\{0, \min_{\substack{j\in \{1,\ldots,p\} \\ k \in \{p+1,\ldots,n\}}} \left(\phi_j^*(u_j) - \phi_k^*(1-u_k) \right) \right\},$$
where $C^\prime$ equals
\[
     \sum_{\mathbf{z}\in{\{0,1\}}^{\{p+1,\ldots,n\}}}(-1)^{\Sigma z} C(\phi_1(u_1),\ldots,\phi_p(u_p),\phi_{p+1}(1-u_{p+1})^{z_{p+1}},\ldots, \phi_n(1-u_n)^{z_n}).
\]
\end{theorem}
Here, in the last displayed equation, we use one of the standard notations for the inclusion-exclusion principle formula. So, $\mathbf{z}$ denotes a vector of zeros and ones whose entries are indexed by $i\in\{p+1,\ldots,n\}$, and
\[
    \phi_{i}(1-u_{i})^{z_{i}}=\left\{
                                \begin{array}{ll}
                                  1, & \hbox{if $z_i=0$;} \\
                                  \phi_{i}(1-u_{i}), & \hbox{if $z_i=1$.}
                                \end{array}
                              \right.
\]
The notation above means that we are taking into account all the possible vectors of the kind.

\begin{proof}
First recall Formula \cite[(4.4)]{DuOmOrRu}
\[
\begin{split}
    & T_{\bf \Phi}(C)(\bu)= \\
     & \sum_{K \subseteq S} C\left(\bPhi(\bu)_{ \{ 1:p\} \cup K^c}\right) \max \left\{ 0, \min \left\{ \bPhi^*(\bu)_{\left\{ 1:p\right\} \cup K}\right\} -\max \left\{ \bPhi^*(\bu)_{K^c}\right\} \right\},
\end{split}
\]
where $S=\{p+1,\ldots,n \}.$

Take a single term of $T$ for some $K\subseteq S$ with $|K|=r$, and denote it by $T_K$. Without loss of generality we may assume that $K=\{p+1,\ldots,p+r\}$ and $K^c=\{p+r+1,\ldots,n\}.$ To somewhat simplify the notation in what follows we introduce
$$
\widetilde{C}:= C\left(\phi_1(u_1),\ldots,\phi_p(u_p),1,\ldots,1,\phi_{p+r+1}(u_{p+r+1}),\ldots,\phi_n(u_n) \right)
$$
and $\phi_i:=\phi_i(u_i)$.
Then $T_K$ equals
\begin{align*}
T_K &= \widetilde{C} \left\{ 0, \min_{j\in \{1,\ldots,p+r\}} \phi^*_j - \max_{k\in \{p+r+1,\ldots,n\} } \phi^*_k \right\}\\
&= \widetilde{C} \max \left\{ 0, \min_{\substack{j\in \{1,\ldots,p+r\} \\ k \in \{ p+r+1,\ldots n\}}} \left( \phi^*_j -\phi^*_k \right) \right\}.
\end{align*}
We start by transforming the variables with indices in $K$.
Apply transformation $u_{p+1} \mapsto 1-u_{p+1}$ to term $T_K$ to get:
\begin{align*}
T_{K,\{p+1\}} := &\widetilde{C} \max \left\{ 0, \min
\left( \phi^*_j -\phi^*_k \right) \right\} \\
&  -\widetilde{C} \max \left\{ 0, \min
 \left\{ \phi^*_j -\phi^*_k, \, \phi^*_{p+1}(1-u_{p+1}) -\phi^*_k  \right\} \right\} \\
=&  \widetilde{C} \max \left\{ 0, \min \left\{ \phi^*_j -\phi^*_k, \,\phi^*_j - \phi^*_{p+1}(1-u_{p+1}) \right\}\right\},
\end{align*}
where the three inside minima are all taken over $j\in \{1,\ldots,p+r\} \setminus \{p+1\}$ and $k \in \{ p+r+1,\ldots n\}$. (Note that we have used Rule \eqref{rule} here and will be doing so similarly in what follows.)

Next, apply transformation $u_{p+2}\mapsto 1-u_{p+2}$ to term $T_{K,\{p+1\}}$ to get:
\begin{align*}
T_{K,\{p+1,p+2\}} := & \widetilde{C} \max \left\{ 0, \min  \left\{\phi^*_j-\phi^*_k, \, \phi^*_j- \phi^*_{p+1}(1-u_{p+1}) \right\}   \right\} \\
 & - \widetilde{C}  \max \left\{ 0, \min \left\{\phi^*_j-\phi^*_k, \, \phi^*_j- \phi^*_{p+1}(1-u_{p+1}), \right. \right. \\  & \hspace{20pt} \left.\left. \phi^*_{p+2}(1-u_{p+2})- \phi^*_k, \phi^*_{p+2}(1-u_{p+2}) - \phi^*_{p+1}(1-u_{p+1}) \right\}   \right\} \\
 =& \widetilde{C} \max \left\{ 0, \min  \left\{\phi^*_j-\phi^*_k, \, \phi^*_j- \phi^*_{p+1}(1-u_{p+1}), 
\phi^*_j- \phi^*_{p+2}(1-u_{p+2}) \right\} \right\},
\end{align*}
where the inside minima are taken over $j\in \{1,\ldots,p+r\} \setminus \{p+1,p+2\}$ and $k \in \{ p+r+1,\ldots n\}$.

We proceed inductively on the index of transformation variable up to $p+r$ to get the intermediate result
$$T_{K,K}:=\widetilde{C} \max \left\{ 0, \min_{\substack{j\in \{1, \ldots, p\} \\ k \in \{p+r+1,\ldots, n\} \\ \ell \in \{p+1,\ldots,p+r\}}} \left\{ \phi^*_j-\phi^*_k, \phi^*_j- \phi^*_\ell(1-u_\ell) \right\}  \right\}.$$

Now, we start performing the transformations of variables with indices in $K^c$. So far, the factor corresponding to the copula $C$ has not been changing. From now on it will change, but only in indices from $p+r+1$ on. To somewhat simplify the notation we introduce:
\[
    \widecheck{C}(u_{p+r+1},\ldots,u_n)=C(\phi_1,\ldots, \phi_p,1,\ldots,1, u_{p+r+1},\ldots,u_n)
\]
First, apply transformation $u_{p+r+1} \mapsto 1- u_{p+r+1}$ on the last term $T_{K,K}$ above:
\begin{align*}
T_{K,K}^{ \{p+r+1\}}  :=&  \widecheck{C}(1, \phi_{p+r+2},\ldots,\phi_n)   \max \left\{ 0, \min \left\{ \phi^*_j-\phi^*_k,\phi^*_j- \phi^*_\ell(1-u_\ell),\phi^*_j-1 \right\} \right\}\\
&-  \widecheck{C}( \phi_{p+r+1}(1-u_{p+r+1}),\phi_{p+r+2},\ldots,\phi_n) \times \\
&  \max \left\{ 0, \min \left\{ \phi^*_j-\phi^*_k, \phi^*_j- \phi^*_\ell(1-u_\ell) , \phi^*_j-\phi^*_{p+r+1}(1-u_{p+r+1}) \right\}  \right\},
\end{align*}
where the two inside minima are to be taken over $j\in \{1, \ldots, p\}$, $k \in \{p+r+2,\ldots, n\}$, and $\ell \in \{p+1,\ldots,p+r\}$. Note that $\phi^*_j -1 <0$, so that only the second term of the expression above stays. By applying the transformation $u_{p+r+2} \mapsto 1- u_{p+r+2}$ to the expression under consideration, we get
\begin{align*}
T_{K,K}^{ \{ p+r+1,p+r+2\}} := & - \widecheck{C}(\phi_{p+r+1}(1-u_{p+r+1}),1,\phi_{p+r+3},\ldots, \phi_n) \times \\
 &   \max \left\{ 0, \min \left\{ \phi^*_j-\phi^*_k, \phi^*_j- \phi^*_\ell(1-u_\ell) , \phi^*_j-1 \right\}  \right\} \\
 & + \widecheck{C}(\phi_{p+r+1}(1-u_{p+r+1}),\phi_{p+r+2}(1-u_{p+r+2}),\phi_{p+r+3},\ldots, \phi_n) \times
\end{align*}
\[
\max \left\{ 0, \min \left\{ \phi^*_j-\phi^*_k, \phi^*_j- \phi^*_\ell(1-u_\ell) , \phi^*_j-\phi^*_{p+r+1}(1-u_{p+r+1}), \, \phi^*_j-\phi^*_{p+r+2}(1-u_{p+r+2}) \right\}  \right\},
\]
where the inside minima are taken over $j\in \{1, \ldots, p\},\ k \in \{p+r+3,\ldots, n\}$, and $\ell \in \{p+1,\ldots,p+r\}$. Similarly as above we see that only the second term stays. Again, we proceed inductively to get the final result for the transformed term:
$$
T_{K,K}^{K^c}:= (-1)^{n-(p+r)} C(\phi_1,\ldots,\phi_p,1,\ldots,1,\phi_{p+r+1}(1-u_{p+r+1}),\ldots, \phi_n(1-u_n))\times
$$
$$
\max \left\{0, \min_{\substack{j \in \{1,\ldots,p\} \\ k\in \{p+1,\ldots, n\}}} \left(\phi^*_j-\phi^*_k(1-u_k) \right) \right\}.
$$
The desired reflected maxmin copula $\dot{T}$ is then the sum of all the terms of the kind. We first observe that the second factor of this term (i.e.\ the one that does not correspond to $C$), is independent of the choice of $K$, so that we can take it out from the summation. The sum of the first factors, on the other hand, follows a certain well-known pattern called inclusion-exclusion principle. This factor of each term is of the form
\[
    C(\phi_1(u_1)),\ldots,\phi_p(u_p),\phi_{p+1}(1-u_{p+1})^{z_{p+1}}, \phi_n(\ldots,(1-u_n)^{z_n}),
\]
where $z$ is any function $z:\{p+1,\ldots,n\}\mapsto \{0,1\}$, and $(1-u_{p+1})^{z_{j}}$ means
\[
    (1-u_{p+1})^{z_{j}}=\left\{
                          \begin{array}{ll}
                            1-u_{p+1}, & \hbox{if $z_j=1$;} \\
                            1, & \hbox{if $z_j=0$.}
                          \end{array}
                        \right.
\]
The sign of this factor equals $-1$ if we have an odd number of inclusions (i.e.\ 1's in the row) and $+1$ otherwise, so it equals to $(-1)^{\Sigma z}$. This completes the proof.
\end{proof}

It remains to translate the formula for the reflected maxmin copula in terms of the generators introduced in \eqref{multi_inverse_generators} of Section \ref{sec:multi3} to get the general case of Formula \eqref{sqi_n3p1}. We follow exactly the same computations as in Section \ref{sec:multi3} to get the following theorem, where $\dot{C}$ denotes copula $C$ transformed with a flip in each of variables with indices $i=p+1,\ldots,n$.

\begin{theorem}\label{thm:multi}
The reflected maxmin $n$-copula $\dot{T}$ with $p$ maxima and $n-p$ minima, corresponding to generating functions $\mathbf{f}$ and reflected copula $\dot{C}$, is of the form
$$\dot{T_{\mathbf{f}}}(\dot{C})({\bf u})= \frac{\dot{C}(\widehat{f}_1(u_1),\cdots,(\widehat{f}_n(u_n)))} {\widehat{f}_1(u_1),\cdots,\widehat{f}_n(u_n)}\times$$
$$
\max \left\{0, \min_{\substack{j\in \{1,\ldots,p\} \\ k \in \{p+1,\ldots,n\}}} \left( \left(u_ju_k-f_j(u_j)f_k(u_k) \right)\prod_{\substack{i\in \{1,\ldots,n\} \\ i \neq j,k}} \widehat{f}_i(u_i)\right) \right\}.$$
\end{theorem}

\begin{center}
  \textbf{Appendix}
\end{center}

Let us give a rough explanation of how the three-dimensional scatterplot of a $3$-copula $C(x,y,z)$ is generated. It follows roughly the following process:
\begin{itemize}
	\item Determine all three marginal 2-copulas and their respective contourplots.
	\item Check for existence of singularities of marginal 2-copulas and respective masses. If marginals are without singularities, copula $C(x,y,z)$ doesn't have a 1-dimensional singularity.
	\item Split the sampling of 3-copula into regions where marginal $C(x,y,1)$ is continuously distributed and where it is singular.
	\item Calculate $\frac{\partial^2 C}{\partial x \partial y}$ and determine maximal subregions of $\mathbb{I}^3$ of continuity.
	\item The most difficult part of the process is the calculation of pseudo-inverse of $\frac{\partial^2 C}{\partial x \partial y}$ with respect to variable $z$. Beside inverting function rules, it is also necessary to transform the boundary conditions where these rules apply.
	\item At discontinuities of $\frac{\partial^2 C}{\partial x \partial y}$ singular components occur. The size of the gap implies the amount of mass gathered on respective singular component.
	\item When the calculation of pseudo-inverse of $\frac{\partial^2 C}{\partial x \partial y}$ is complete, generate a sample set of say 5000 pairs $(x_i,y_i)$ distributed with respect to marginal 2-copula $C(x,y,1)$ and an independent sample $z_i$ uniformly distributed on $\mathbb{I}$.
	\item Merge sets $\{(x_i,y_i)\}$ and $\{z_i\}$ into 3-tuples $(x_i,y_i,z_i)$.
	\item For every 3-tuple $(x_i,y_i,z_i)$ determine the region that $(x_i,y_i)$ belongs to and apply corresponding pseudo-inverse function rule to $z_i$.
	\item The process is done manually with lots of careful considerations which doesn't give rise to hope of automatization.
\end{itemize}


\begin{thebibliography}{00}



\bibitem{BeBoCiSaPlSa} L.\ B{\v{e}}hounek, U.\ Bodenhofer, P.\ Cintula, S.\ Saminger-Platz, P.\ Sarkoci, \textsl{Graded dominance and related graded properties of fuzzy connectives}, Fuzzy Sets and Systems, \textbf{262} (2015), 78--101.

\bibitem{ChDuMu} U.\ Cherubini, F.\ Durante, S. Mulinacci, (eds.), \textsl{Marshall--{O}lkin {D}istributions -- {A}dvances in {T}heory and {A}pplications}, Springer Proceedings in Mathematics \& Statistics, Springer International Publishing, 2015.

\bibitem{ChMu}U.\ Cherubini, S.\ Mulinacci, \emph{Systemic risk with exchangeable contagion: application to the European banking system}, ArXiv e-prints, 2015.

\bibitem{DuFeSaUbFl} F.~Durante, J.\ Fern\'andez S\'anchez, M.\ \'Ubeda Flores, \textsl{Bivariate copulas generated by perturbarion}, Fuzzy Sets and Systems, \textbf{228} (2013), 137--144.

\bibitem{DuGiMa1} F.~Durante, S.\ Girard, G.\ Mazo, \textsl{Copulas Based on Marshall-Olkin Machinery}, Chapter 2 in: U. Cherubini, F. Durante, S. Mulinacci (Eds.), Marshall-Olkin Distributions -- Advances in Theory and Practice, in: Springer Proceedings in Mathematics \& Statistics, Springer, 2015, pp. 15--31.

\bibitem{DuGiMa} F.~Durante, S.\ Girard, G.\ Mazo, \textsl{{M}arshall--{O}lkin type copulas generated by a global shock}, J.\ Comput.\ Appl.\ Math., \textbf{296} (2016), 638--648.

\bibitem{DuJa} F.\ Durante, P.\ Jaworski, \textsl{A new characterization of bivariate copulas}, Comm.\ in Stats.\ Theory and Methods, \textbf{39} (2010), 2901--2912.

\bibitem{DuKoMeSe} F.~Durante, A.~Kolesarov\`{a}, R.~Mesiar, C.~Sempi, \textsl{Semilinear copulas}, Fuzzy Sets and Systems, \textbf{159} (2008), 63--76.

\bibitem{DuMePaSe} F.~Durante, R.\ Mesiar, P.\ L.\ Papini, C.~Sempi, \textsl{2-Increasing binary aggregation operators}, Information Sciences, \textbf{177} (2007), 111--129.

\bibitem{DuOmOrRu} F.~Durante, M.\ Omladi\v{c}, L.\ Ora\v{z}em, N.\ Ru\v{z}i\'{c}, \textsl{Shock models with dependence and asymmetric linkages}, Fuzzy Sets and Systems, \textbf{323} (2017), 152--168.

\bibitem{DuSe} F.~Durante, C.~Sempi, Principles of Copula Theory, CRC/Chapman \& Hall, Boca Raton (2015).

\bibitem{FrNe} G.\ A.\ Fredricks, R.\ B.\ Nelsen, \textsl{On the relationship between Spearman’s rho and Kendall’s tau for pairs of continuous random variables}, Journal of Statistical Planning and Inference \textbf{137} (2007), 2143--2150

\bibitem{FrGr} B.\ E.\ Fristedt, L.\ F.\ Gray, \textsl{AModern Approach to Probability Theory}, Probability and Its Applications, Birkh{\"a}user, Boston, 2013.

\bibitem{GeNe} C.\ Genest, J.\ Ne\v{s}lehov\'a, \textsl{Assessing and Modeling Asymmetry in Bivariate Continuous Data}, In: P.\ Jaworski, F.\ Durante, W.K.\ H\"ardle, (eds.), {C}opulae in {M}athematical and {Q}uantitative {F}inance, Lecture Notes in Statistics, Springer Berlin Heidelberg, (2013), 152--16891--114.

\bibitem{Hu} T.\ E.\ Huillet, \textsl{Stochastic species abundance models involving special copulas}, Phisica A, \textbf{490} (2018) 77--91



\bibitem{JwBaDeMe14}  T.\ Jwaid, B.\ De Baets, H. De Meyer, \textsl{Ortholinear and paralinear semi-copulas}, Fuzzy Sets and Systems, \textbf{252} (2014), 76--98.

\bibitem{JwBaDeMe15}  T.\ Jwaid, B.\ De Baets, H. De Meyer, \textsl{Semiquadratic copulas based on horizontal and vertical interpolation}, Fuzzy Sets and Systems, \textbf{264} (2015), 3--21.

\bibitem{KlLiMePa}  E.\ P.\ Klement, J. Li, R. Mesiar, E. Pap, \textsl{Integrals based on monotone set functions}, Fuzzy Sets and Systems, \textbf{281} (2015), 3--21.

\bibitem{KlMeSpSt} E.\ P.\ Klement, R.\ Mesiar, F.\ Spizzichino, A.\ Stup{\v{n}}anov{\'a}, \textsl{Universal integrals based on copulas}, Fuzzy Optim.\ Decis.\ Mak., \textbf{13}, No.\ 3, (2014), 273--286.


\bibitem{KoBuKoMoOm} D.\ Kokol Bukov\v{s}ek, T.\ Ko\v{s}ir, B.\ Moj\v{s}kerc, and M.\ Omladi\v{c}, \textsl{Non-exchangeability of copulas arising from shock models}, J. Comput. Appl. Math. \textbf{ 358} (2019), 61-83.

\bibitem{KoOm} T.\ Ko\v{s}ir, M.\ Omladi\v{c}, \textsl{Reflected maxmin copulas and modelling quadrant subindependence}, Fuzzy Sets and Systems, available online, in Press.

\bibitem{LiMcNe} F.\ Lindskog, A.\ J.\ McNeil, \textsl{Common Poisson Shock Models: Applications to insurance and credit risk modelling}, ASTIN Bulletin, \textbf{33}, No. 2, (2003), 209--238.

\bibitem{Mars} A.~W.~Marshall, \textsl{Copulas, marginals, and joint distributions}, in: L.~R\"{u}schendorf, B.~Schweitzer, M.~D.~Taylor  (eds.), Distributions with Fixed Marginals and Related Topics in LMS, Lecture Notes -- Monograph Series, vol.~ \textbf{28}, 1996, 213--222.

\bibitem{MaOl} A.~W.~Marshall, I.~Olkin, \textsl{A multivariate exzponential distributions}, J.~Amer.~Stat.~Assoc., \textbf{62}, (1967), 30--44.


\bibitem{MeKoKo}  R.\ Mesiar, M.\ Komorn\'{i}kov\'{a}, J.\ Komorn\'{i}k, \textsl{Perturbation of bivariate copulas}, Fuzzy Sets and Systems, \textbf{268} (2015), 127--140.

\bibitem{Mu} S.\ Mulinacci. \emph{Archimedean-based Marshall-Olkin distributions and related dependence structures}, Methodol.\ Comput.\ Appl.\ Probab., {\bf 20} (2018), 205--236.



\bibitem{Nels} R.\ B.\ Nelsen, An introduction to copulas, 2nd edition, Springer-Verlag, New York (2006).


\bibitem{OmRu}  M.\ Omladi\v{c}, N.\ Ru\v{z}i\'{c}, \textsl{Shock models with recovery option via the maxmin copulas}, Fuzzy Sets and Systems, \textbf{284} (2016), 113--128.

\bibitem{RoLaUbFl}  J.\ A.\ Rodr\'iguez-Lallena, M.\ \'Ubeda-Flores, \textsl{A new class of bivariate copulas}, Stat.\ \& Probab.\  Lett., \textbf{66} (2004), 315--325.


\bibitem{Skla} A.\ Sklar, Fonctions de r\'{e}partition \`{a} $n$ dimensions et leurs marges, Publ.\ Inst.\ Stat.\ Univ.\ Paris \textbf{8} (1959) 229--231.

\bibitem{math} Wolfram Research, Inc. Mathematica, Version {\bf 11}, Champaign, IL, 2017.

\end{thebibliography}
\end{document}